\documentclass{emsprocart}

\usepackage{amssymb}
\usepackage[all]{xy}

\entrymodifiers={+!!<0pt,\fontdimen22\textfont2>}

\contact[peter.jorgensen@ncl.ac.uk]{Peter J\o rgensen, School of
  Mathematics and Statistics, Newcastle University, Newcastle upon
  Tyne NE1 7RU, United Kingdom}

\numberwithin{equation}{section}

\newtheorem{theorem}{Theorem}[section]

\newtheorem{lemma}[theorem]{Lemma}
\newtheorem{proposition}[theorem]{Proposition}

\newtheorem{setup}[theorem]{Setup}

\theoremstyle{definition}
\newtheorem{definition}[theorem]{Definition}
\newtheorem{remark}[theorem]{Remark}
\newtheorem{problem}[theorem]{Problem}

\newcommand{\BP}{{\mathbb P}}
\newcommand{\BQ}{{\mathbb Q}}

\newcommand{\BZ}{{\mathbb Z}}

\newcommand{\sD}{\mathsf D}

\newcommand{\sK}{\mathsf K}

\newcommand{\sT}{\mathsf T}

\newcommand{\amp}{\operatorname{amp}}

\newcommand{\Chain}{\operatorname{C}}

\newcommand{\Dcat}{\sD}
\newcommand{\Dcatc}{\sD^{\operatorname{c}}}

\newcommand{\Dcatf}{\sD^{\operatorname{f}}}

\newcommand{\dual}{\operatorname{D}}

\newcommand{\Ext}{\operatorname{Ext}}

\newcommand{\Homology}{\operatorname{H}}

\newcommand{\Hom}{\operatorname{Hom}}

\newcommand{\LTensor}{\stackrel{\operatorname{L}}{\otimes}}
\newcommand{\opp}{\operatorname{o}}

\newcommand{\RHom}{\operatorname{RHom}}

\title[Calabi-Yau categories and Poincar\'{e} duality
spaces]{Calabi-Yau categories and Poincar\'{e} du\-a\-li\-ty spaces} 

\author[Peter J\o rgensen]{Peter J\o rgensen}

\begin{document}

\begin{abstract}  
  The singular cochain complex of a topological space is a classical
  object.  It is a Differential Graded algebra which has been studied
  intensively with a range of methods, not least within rational
  homotopy theory.
  
  More recently, the tools of Auslander-Reiten theory have also been
  applied to the singular cochain complex.  One of the highlights is
  that by these methods, each Poincar\'{e} duality space gives rise to
  a Calabi-Yau category.  This paper is a review of the theory.
\end{abstract}



\maketitle

\section{Introduction}

Finite dimensional algebras over a field are classical, well studied
mathematical objects.  Their representation theory is a particularly
large and active area which has inspired a number of powerful
mathematical techniques, not least Auslander-Reiten theory which is a
beautiful and effective set of tools and ideas.  See Appendix
\ref{app:AR} and the references listed there for an introduction.

It seems reasonable to look for applications of Auslander-Reiten (AR)
theory to areas outside representation theory.  Specifically, let $X$
be a topological space.  The singular cochain complex $\Chain^*(X;k)$
with coefficients in a field $k$ of characteristic $0$ is a
Differential Graded algebra which has been studied intensively, in
particular in rational homotopy theory, see \cite{FHTbook}.  For an
introduction to Differential Graded (DG) homological algebra, see
Appendix \ref{app:DG} and the references listed there.  The singular
cohomology $\Homology^*(X;k)$ is defined as the cohomology of the
complex $\Chain^*(X;k)$; it is a graded algebra.  Now let $X$ be
simply connected with $\dim_k \Homology^*(X;k) < \infty$; then
$\Chain^*(X;k)$ is quasi-isomorphic to a DG algebra $R$ with $\dim_k R
< \infty$, and it is natural to try to apply AR theory to $R$.  This
was the subject of \cite{artop}, \cite{arquiv}, and \cite{Schmidt},
and the object of this paper is to review the results of those papers.

Among the highlights is Theorem \ref{thm:Chain_CY} from which comes
the title of the paper.  Consider the derived category of DG
left-$R$-modules, $\Dcat(R)$, which is equivalent to
$\Dcat(\Chain^*(X;k))$ since the two DG algebras are quasi-isomorphic.
The latter category has the full subcategory $\Dcatc(\Chain^*(X;k))$
consisting of compact DG modules; these play the role of finitely
generated representations.  Theorem \ref{thm:Chain_CY} now says that
if $k$ has characteristic $0$, then
\begin{align}
\label{equ:z}
  & \mbox{$\Dcatc(\Chain^*(X;k))$ is an $n$-Calabi-Yau category} \\
\nonumber
  & \mbox{$\Leftrightarrow$ $X$ has $n$-dimensional Poincar\'{e}
    duality over $k$.}
\end{align}

Let me briefly explain the terminology.  A triangulated category
$\sT$, such as for instance $\Dcatc(\Chain^*(X;k))$, is called
$n$-Calabi-Yau if $n$ is the smallest non-negative integer for which
$\Sigma^n$, the $n$th power of the suspension functor, is a Serre
functor, that is, permits natural isomorphisms
\[
  \Hom_k(\Hom_{\sT}(M,N),k) \cong \Hom_{\sT}(N,\Sigma^n M).
\]
The topological space $X$ is said to have $n$-dimensional Poincar\'{e}
duality over $k$ if there is an isomorphism
\[
  \Hom_k(\Homology^*(X;k),k) \cong \Sigma^n \Homology^*(X;k)
\]
of graded left-$\Homology^*(X;k)$-modules.

Examples of $n$-Calabi-Yau categories are higher cluster categories,
see \cite[sec.\ 4]{KellerReiten}, and examples of spaces with
$n$-dimensional Poincar\'{e} duality are compact $n$-dimensional
manifolds.  Equation \eqref{equ:z} provides a link between the
currently po\-pu\-lar theory of Calabi-Yau categories and algebraic
topology.  It also gives a new class of examples of Calabi-Yau
categories which, so far, typically have been exemplified by higher
cluster categories.  The new categories appear to behave very
differently from higher cluster categories, cf.\ Section
\ref{sec:open}, Problem \ref{prob:CY}.

A number of other results are also obtained, not least on the
structure of the AR quiver of $\Dcatc(\Chain^*(X;k))$ which, for a
space with Poincar\'{e} duality, consists of copies of the repetitive
quiver $\BZ A_{\infty}$, see Theorem \ref{thm:Chain_Gamma}.

In a speculative vein, the theory presented here ties in with the
version of non-commutative geometry in which a DG algebra, or more
generally a DG category, is viewed as a non-commutative scheme.  The
idea is to think of the derived category of the DG algebra or DG
category as being the derived category of quasi-coherent sheaves on a
non-commutative scheme (which does not actually exist).  There appear
so far to be no published references for this viewpoint which has been
brought forward by Drinfeld and Kontsevich, but it does seem to call
for a detailed study of the derived categories of DG algebras and DG
categories.  Auslander-Reiten theory is an obvious tool to try, and
\cite{artop}, \cite{arquiv}, and \cite{Schmidt} along with this paper
can, perhaps, be viewed as a first, modest step.

As indicated, the paper is a review.  The results were known
previously, the main references being \cite{artop}, \cite{arquiv}, and
\cite{Schmidt}; more details of the origin of individual results are
given in the introductions to the sections.  There is no claim to
o\-ri\-gi\-na\-li\-ty, except that some of the proofs are new.  It is
also the first time this material has appeared together.

Most of the paper is phrased in terms of the DG algebra $R$ rather
than $\Chain^*(X;k)$, see Setup \ref{set:blanket}.  This is merely a
notational convenience: $R$ and $\Chain^*(X;k)$ are quasi-isomorphic,
so have equivalent derived categories.  Hence, all results about the
derived category of $R$ also hold for the derived category of
$\Chain^*(X;k)$.  The paper is organized as follows:

Some background on DG homological algebra and AR theory is collected
in two appendices, \ref{app:DG} and \ref{app:AR}.

Section \ref{sec:cochain} gives some preliminary results on cochain DG
algebras and their DG modules.  The main result is Theorem
\ref{thm:Gorenstein} which gives a number of alternative descriptions
of when $R$ is a so-called Gorenstein DG algebra.  The importance of
this condition is that $\Chain^*(X;k)$ is Gorenstein precisely when
$X$ has Poincar\'{e} duality.

Section \ref{sec:AR} studies the existence of AR
triangles in the category $\Dcatc(R)$, which turns out to be
equivalent to $R$ being Gorenstein by Theorem \ref{thm:R}.

Section \ref{sec:local} considers the local structure of the AR quiver
$\Gamma$ of $\Dcatc(R)$.  If $R$ is Gorenstein with $\dim_k
\Homology\!R \geq 2$, then Theorem \ref{thm:quiver_local_structure}
shows that each component of $\Gamma$ is isomorphic to $\BZ
A_{\infty}$.

Section \ref{sec:global} reports on work by Karsten Schmidt.  It looks
at the global structure of $\Gamma$ where the results are so far less
conclusive.  If $\dim_k \Homology\!R = 2$, then $\Gamma$ has precisely
$d-1$ components isomorphic to $\BZ A_{\infty}$, where $d = \sup \{\,
i \,|\, \Homology^i\!R \neq 0 \,\}$.  On the other hand, if $R$ is
Gorenstein with $\dim_k \Homology\!R \geq 3$, then $\Gamma$ has
infinitely many components, and if $\dim_k \Homology^e\!R \geq 2$ for
some $e$, then it is even possible to find families of distinct
components indexed by projective manifolds, and these manifolds can be
of arbitrarily high dimension.

Section \ref{sec:topology} makes explicit the highlights of the theory
for the algebras $\Chain^*(X;k)$.

Section \ref{sec:open} is a list of open problems.

\medskip
\noindent
{\bf Acknowledgement.}
Some of the results of this paper, not least the ones of Section
\ref{sec:global}, are due to Karsten Schmidt.  I thank him for a
number of communications on his work, culminating in \cite{Schmidt}.
I thank Henning Krause, Andrzej Skowronski, and the referee for
comments to a previous version of the paper.  I am grateful to Andrzej
Skowronski for the very succesful organization of ICRA XII in Torun,
August 2007, and for inviting me to submit this paper to the ensuing
volume ``Trends in Representation Theory of Algebras and Related
Topics''.

\section{Cochain Differential Graded algebras}
\label{sec:cochain}

This section provides some results on cochain Differential Graded (DG)
algebras, not least on the ones which are Gorenstein.  The results
first appeared in \cite{artop}, except Lemma \ref{lem:inf} which is
\cite[lem.\ 1.5]{FJcochain} and Theorem \ref{thm:k} which is
\cite[cor.\ 3.12]{Schmidt}.

For background and terminology on DG algebras and their derived
categories, see Appendix \ref{app:DG}.

\begin{setup}
\label{set:blanket}
In Sections \ref{sec:cochain} through \ref{sec:open}, $k$ is a
field and $R$ is a DG algebra over $k$ which has the form 
\[
  \cdots \rightarrow 0 \rightarrow k \rightarrow 0 \rightarrow R^2
  \rightarrow R^3 \rightarrow \cdots,
\]
that is, $R^{<0} = 0$, $R^0 = k$, and $R^1 = 0$.  It will be assumed
that $\dim_k R < \infty$, and throughout,
\[
  d = \sup R
\]
where $\sup$ is as in Definition \ref{def:sup_and_inf}.
\end{setup}

Note that either $d = 0$, in which case $R$ is quasi-isomorphic to
$k$, or $d \geq 2$.

\begin{remark}
If $X$ is a simply connected topological space with $\dim_k
\Homology^*(X;k) < \infty$ and $k$ has characteristic $0$, then
$\Chain^*(X;k)$ is quasi-isomorphic to a DG algebra $R$ satisfying the
conditions of Setup \ref{set:blanket} by \cite[exa.\ 6, p.\ 
146]{FHTbook}.  This means that the derived categories of
$\Chain^*(X;k)$ and $R$ are equivalent, and hence, all results about
the derived category of $R$ also hold for the derived category of
$\Chain^*(X;k)$.

The highlights of the theory will be made explicit for $\Chain^*(X;k)$
in Section \ref{sec:topology}.
\end{remark}

\begin{proposition}
\label{pro:DcinDf}
The full subcategory $\Dcatc(R)$ of compact objects of the derived
category $\Dcat(R)$ is contained in $\Dcatf(R)$, the full subcategory
of $\Dcat(R)$ of objects with $\dim_k \Homology\!M < \infty$.
\end{proposition}

\begin{proof}
The DG module ${}_{R}R$ is in $\Dcatf(R)$ by assumption, and
$\Dcatc(R)$ consists of the DG modules which are finitely built from
it, cf.\ Definition \ref{def:D}, so it follows that $\Dcatc(R)$ is
contained in $\Dcatf(R)$.
\end{proof}

\begin{proposition}
\label{pro:DcandDfKrullSchmidt}
The triangulated categories $\Dcatf(R)$ and $\Dcatc(R)$ have finite
dimensional Hom spaces and split idempotents.

Consequently, $\Dcatf(R)$ and $\Dcatc(R)$ are Krull-Schmidt
categories. 
\end{proposition}

\begin{proof}
If $M$ is in $\Dcatf(R)$, then it is finitely built from ${}_{R}k$ in
$\Dcat(R)$, see Remark \ref{rmk:Df}.  So to see that $\Dcatf(R)$
has finite dimensional Hom spaces, it is enough to see that
$\Hom_{\Dcat(R)}(\Sigma^i k,k)$ is finite dimensional for each $i$,
where $\Sigma$ denotes the suspension functor of $\Dcat(R)$.

Let $F$ be a minimal semi-free resolution of ${}_{R}(\Sigma^i k)$;
then
\begin{align*}
  \Hom_{\Dcat(R)}(\Sigma^i k,k)
  & \stackrel{\rm (a)}{\cong} \Homology^0 \RHom_R(\Sigma^i k,k) \\
  & \stackrel{\rm (b)}{\cong} \Homology^0 \Hom_R(F,k) \\
  & \stackrel{\rm (c)}{\cong} \Hom_{R^{\natural}}(F^{\natural},k^{\natural})^0\\
  & = (*)
\end{align*}
where (a) is by Definition \ref{def:Hom_and_Tensor} and (b) and (c)
are by Lemma \ref{lem:semi-free}, (2) and (5).  However, Lemma
\ref{lem:semi-free}(3) says that $F^{\natural} \cong \bigoplus_{j \leq
i} \Sigma^j(R^{\natural})^{(\beta_j)}$ with the $\beta_j$ finite,
where superscript $(\beta)$ indicates the direct sum of $\beta$ copies
of the module, and so
\[
  (*) \cong k^{(\beta_0)}.
\]
This is finite dimensional.

Since $\Dcatc(R)$ is contained in $\Dcatf(R)$ by Proposition
\ref{pro:DcinDf}, it follows that $\Dcatc(R)$ also has finite
dimensional Hom spaces.

Idempotents split in both $\Dcatf(R)$ and $\Dcatc(R)$ since
by \cite[prop.\ 3.2]{BN} they split already in $\Dcat(R)$ because this
is a triangulated category with set indexed co\-pro\-ducts.

By \cite[p.\ 52]{Ringel}, both $\Dcatf(R)$ and $\Dcatc(R)$ are
Krull-Schmidt categories.
\end{proof}

Recall from Definition \ref{def:sup_and_inf} the notion of $\inf$ of a
DG module, and from Definition \ref{def:DG} that $R^{\opp}$ is the
opposite DG algebra of $R$ and that DG left-$R^{\opp}$-modules can be
viewed as DG right-$R$-modules.

\begin{lemma}
\label{lem:inf}
Let $M$ be in $\Dcatf(R^{\opp})$ and let $N$ be in $\Dcatf(R)$.  Then 
\[
  \inf(M \LTensor_R N) = \inf M + \inf N.
\]
\end{lemma}

\begin{proof}
If $M$ or $N$ is isomorphic to zero then the equation reads $\infty
= \infty$, so let me assume not.  Then $i = \inf M$ and $j = \inf N$
are integers.

Lemma \ref{lem:truncations}(1) says that $M$ can be replaced with a
quasi-isomorphic DG mo\-du\-le which satisfies $M^{\ell} = 0$ for
$\ell < i$.

Lemma \ref{lem:semi-free}(3) says that $N$ has a semi-free resolution
$F$ which satisfies that $F^{\natural} \cong \bigoplus_{\ell \leq -j}
\Sigma^{\ell} (R^{\natural})^{(\beta_{\ell})}$, and it follows that
$(M \otimes_R F)^{\natural} \cong \bigoplus_{\ell \leq -j}
\Sigma^{\ell} (M^{\natural})^{(\beta_{\ell})}$.

Since $M^{\ell} = 0$ for $\ell < i$, this implies that $(M \otimes_R
F)^{\ell} = 0$ for $\ell < i+j$.  In particular, $\inf(M \otimes_R F)
\geq i+j$ whence
\begin{equation}
\label{equ:h}
  \inf(M \LTensor_R N) \geq i+j = \inf M + \inf N.
\end{equation}

Conversely, to give a morphism of DG left-$R$-modules $\Sigma^{-j}R
\rightarrow N$ is the same thing as to give the image $z$ of
$\Sigma^{-j}(1_R)$, and $z$ is a cycle in $N^j$.  Since
$\Homology^j(\Sigma^{-j}R) \cong \Homology^0\!R \cong k$, the induced
map $\Homology^j(\Sigma^{-j}R) \rightarrow \Homology^j\!N$ is just the
map $k \rightarrow \Homology^j\!N$ sending $1_k$ to the cohomology
class of $z$.  Hence, picking cycles $z_{\alpha}$ whose cohomology
classes form a $k$-basis of $\Homology^j\!N$ and constructing a
morphism $\Sigma^{-j}R^{(\beta)} \rightarrow N$ by sending the
elements $\Sigma^{-j}(1_R)$ to the $z_{\alpha}$ gives that the induced
map $\Homology^j(\Sigma^{-j}R^{(\beta)}) \rightarrow \Homology^j\!N$ is
an isomorphism.  Complete to a distinguished triangle
\begin{equation}
\label{equ:i}
  \Sigma^{-j}R^{(\beta)} \rightarrow N \rightarrow N^{\prime\prime} \rightarrow;
\end{equation}
since $\Homology^{j+1}(\Sigma^{-j}R^{(\beta)}) \cong
\Homology^1(R^{(\beta)}) = 0$, the long exact cohomology sequence
shows
\begin{equation}
\label{equ:k}
  \inf N^{\prime\prime} \geq j+1.
\end{equation}

Tensoring the distinguished triangle \eqref{equ:i} with
$M$ gives 
\[
  \Sigma^{-j}M^{(\beta)} \rightarrow M \LTensor_R N
    \rightarrow M \LTensor_R N^{\prime\prime} \rightarrow
\]
and the long exact cohomology sequence of this contains
\begin{equation}
\label{equ:l}
  \Homology^{i+j-1}(M \LTensor_R N^{\prime\prime})
    \rightarrow \Homology^{i+j}(\Sigma^{-j}M^{(\beta)})
    \rightarrow \Homology^{i+j}(M \LTensor_R N).
\end{equation}
The inequality \eqref{equ:h} can be applied to $M$ and $N^{\prime\prime}$;
because of the inequality \eqref{equ:k}, this gives $\inf(M \LTensor_R
N^{\prime\prime}) \geq i+j+1$ so the first term of the exact sequence
\eqref{equ:l} is zero.  The second term is
$\Homology^{i+j}(\Sigma^{-j}M^{(\beta)}) \cong
\Homology^i(M^{(\beta)})$ which is non-zero since $i = \inf M$.  It
follows that the third term is non-zero, so
\[
  \inf(M \LTensor_R N) \leq i+j = \inf M + \inf N.
\]
Combining with the inequality \eqref{equ:h} proves the lemma.
\end{proof}

\begin{definition}
The DG algebra $R$ is said to be Gorenstein if it satisfies the
equivalent conditions of the following theorem.
\end{definition}

In the theorem, recall from Definition \ref{def:dual} that $\dual(-) =
\Hom_k(-,k)$.

\begin{theorem}
\label{thm:Gorenstein}
The following conditions are equivalent.
\begin{enumerate}

  \item  There are isomorphisms of $k$-vector spaces
\[
  \Ext_R^i(k,R) \; \cong \;
  \left\{
    \begin{array}{cl}
      k & \mbox{for $i = d$}, \\
      0 & \mbox{otherwise}
    \end{array}
  \right\} \; \cong \;
  \Ext_{R^{\opp}}^i(k,R).
\]

  \item  There are isomorphisms of graded $\Homology\!R$-modules
\[
  \mbox{${}_{\Homology\!R}(\dual\!\Homology\!R) \cong {}_{\Homology\!R}(\Sigma^d \Homology\!R)$
        \;\;and\;\;
        $(\dual\!\Homology\!R)_{\Homology\!R} \cong (\Sigma^d \Homology\!R)_{\Homology\!R}$.}
\]

  \item  There are isomorphisms
\[
  \mbox{${}_{R}(\dual\!R) \cong {}_{R}(\Sigma^d R)$ \;in\; $\Dcat(R)$
        \;\;and\;\;
        $(\dual\!R)_R \cong (\Sigma^d R)_R$ \;in\; $\Dcat(R^{\opp})$.}
\]

  \item  $\dim_k \Ext_R(k,R) < \infty$ and $\dim_k
    \Ext_{R^{\opp}}(k,R) < \infty$. 

  \item  ${}_{R}(\dual\!R)$ is in $\Dcatc(R)$ and $(\dual\!R)_R$ is in
    $\Dcatc(R^{\opp})$. 

\end{enumerate}
\end{theorem}

\begin{proof}
(1)$\Rightarrow$(3).  Let $F$ be a minimal semi-free resolution of
${}_{R}(\dual\!R)$.  Then
\begin{equation}
\label{equ:d}
  \Ext_{R^{\opp}}(k,R)
  \stackrel{\rm (a)}{\cong} \Ext_R(\dual\!R,k)
  = \Homology(\RHom_R(\dual\!R,k))
  \stackrel{\rm (b)}{\cong} \Hom_{R^{\natural}}(F^{\natural},k^{\natural})
\end{equation}
where (a) is by duality and (b) is by Lemma \ref{lem:semi-free}, (2)
and (5).  If the second isomorphism in (1) holds, then this implies
$F^{\natural} \cong \Sigma^d R^{\natural}$.  But then there is clearly
only a single step in the semi-free filtration of $F$, whence $F \cong
{}_{R}(\Sigma^d R)$ so ${}_{R}(\dual\!R) \cong {}_{R}(\Sigma^d R)$,
proving the first isomorphism in (3).  Likewise, the first
isomorphism in (1) implies the second isomorphism in (3).

(3)$\Rightarrow$(2).  This follows by taking cohomology.

(2)$\Rightarrow$(1).  This follows from the Eilenberg-Moore spectral
sequence
\[
  E_2^{pq} = \Ext_{\Homology\!R}^p(k,\Homology\!R)^q
  \Rightarrow \Ext_R^{p+q}(k,R)
\]
which exists by \cite[1.3(2)]{FHTpaper}, and the corresponding
spectral sequence over $R^{\opp}$.

(4)$\Leftrightarrow$(5).  Lemma \ref{lem:semi-free}(3) says that the
semi-free resolution $F$ of ${}_{R}(\dual\!R)$ has $F^{\natural} \cong
\bigoplus_i \Sigma^i(R^{\natural})^{(\beta_i)}$, and Equation
\eqref{equ:d} shows that $\dim_k \Ext_{R^{\opp}}(k,R)$ is the
number of direct summands $\Sigma^i R^{\natural}$.  By Lemma
\ref{lem:semi-free}(4), this number is finite if and only if
${}_{R}(\dual\!R)$ is in $\Dcatc(R)$, so the second condition in (4)
is equivalent to the first condition in (5) and vice versa.

(1)$\Rightarrow$(4) is clear.

(4)$\Rightarrow$(1).  When (4) holds, so does (5) by the previous part
of the proof; hence ${}_{R}(\dual\!R)$ is finitely built from
${}_{R}R$.  Then the canonical morphism
\[
  \RHom_R(k,R) \LTensor_R \dual\!R
  \rightarrow \RHom_R(k,R \LTensor_R \dual\!R)
\]
is an isomorphism, because it clearly is if $\dual\!R$ is replaced
with $R$.  That is,
\begin{equation}
\label{equ:c}
  \RHom_R(k,R) \LTensor_R \dual\!R \cong k.
\end{equation}
Since (4) holds, $\RHom_R(k,R)$ is in $\Dcatf(R^{\opp})$, so Lemma
\ref{lem:inf} applies to the tensor product and gives
\[
  \inf \RHom_R(k,R) + \inf \dual\!R = \inf k = 0
\]
which amounts to
\[
  \inf \RHom_R(k,R) = d.
\]

On the other hand, adjointness gives the first of the next isomorphisms,
\[
  \RHom_k((\dual\!R) \LTensor_R k,k)
  \cong \RHom_R(k,\RHom_k(\dual\!R,k))
  \cong \RHom_R(k,R),
\]
and so
\begin{align*}
  \sup \RHom_R(k,R)
  & = \sup \RHom_k((\dual\!R) \LTensor_R k,k) \\
  & = - \inf((\dual\!R) \LTensor_R k) \\
  & \stackrel{\rm (c)}{=} - \inf \dual\!R - \inf k \\
  & = d
\end{align*}
where (c) is by Lemma \ref{lem:inf} again.  Hence the cohomology of
$\RHom_R(k,R)$ is concentrated in degree $d$, and it is not hard to
show that hence
\[
  \RHom_R(k,R) \cong (\Sigma^{-d}k^{(\beta)})_R
\]
for some $\beta$.  Inserting this into Equation \eqref{equ:c} shows
$\beta = 1$, so
\[
  \RHom_R(k,R) \cong (\Sigma^{-d}k)_R.
\]
This is equivalent to the first isomorphism in (1), and the second one
follows by a symmetric argument.
\end{proof}

\begin{theorem}
\label{thm:k}
If $\dim_k \Homology\!R \geq 2$, then ${}_{R}k$ is not in $\Dcatc(R)$.
\end{theorem}

\begin{proof}
Recall from Definition \ref{def:sup_and_inf} the notion of amplitude
of a DG module.  There is an amplitude inequality $\amp(M \LTensor_R
N) \geq \amp M$ for $M$ in $\Dcatf(R^{\opp})$ and $N$ in $\Dcatc(R)$.
This was first stated in \cite[prop.\ 3.11]{Schmidt}; see
\cite[cor.\ 4.4]{FJcochain} for an alternative proof.

If ${}_{R}k$ were in $\Dcatc(R)$, then this would give $\amp(R
\LTensor_R k) \geq \amp R$, that is, $0 \geq \amp R$ contradicting
$\dim_k \Homology\!R \geq 2$ whereby $R$ must (also) have cohomology in
a degree different from $0$.
\end{proof}

\section{Auslander-Reiten triangles over Differential \\ Gra\-ded algebras}
\label{sec:AR}

In this section, it is proved that the compact derived category
$\Dcatc(R)$ has Auslander-Reiten (AR) triangles if and only if $R$ is
a Gorenstein DG algebra.  In this case, a formula is found for the AR
translation of $\Dcatc(R)$.  These results first appeared in
\cite{artop}.

For background on AR theory, see Appendix \ref{app:AR}.

In the following proposition, note that $\dual\!R \LTensor_R P$
inherits a left-$R$-structure from the DG bi-$R$-module $\dual\!R$ so
$\dual\!R \LTensor_R P$ is in $\Dcat(R)$; see Definition
\ref{def:Hom_and_Tensor}.

\begin{proposition}
\label{pro:AR}
Let $P$ be an indecomposable object of $\Dcatc(R)$.  There is an
AR triangle in $\Dcatf(R)$,
\[
  \Sigma^{-1}(\dual\!R \LTensor_R P) \rightarrow N \rightarrow P \rightarrow.
\]
\end{proposition}

\begin{proof}
Since $P$ is finitely built from ${}_{R}R$, there is a natural
equivalence
\begin{equation}
\label{equ:Serre}
  \dual(\Hom_{\Dcat(R)}(P,-)) \simeq \Hom_{\Dcat(R)}(-,\dual\!R \LTensor_R P),
\end{equation}
since there is clearly such an equivalence if $P$ is replaced with
${}_{R}R$.  By \cite[prop.\ 4.2]{KrauseARZ}, this means that the AR
triangle of the present proposition exists in $\Dcat(R)$.

To complete the proof, observe that the triangle is in fact in
$\Dcatf(R)$: The object $P$ is in $\Dcatc(R)$, so it is in $\Dcatf(R)$
by Proposition \ref{pro:DcinDf}.  Since $R_R$ is in
$\Dcatf(R^{\opp})$, the dual ${}_{R}(\dual\!R)$ is in $\Dcatf(R)$, and
since $P$ is finitely built from $R$, it follows that $\dual\!R
\LTensor_R P$ is also in $\Dcatf(R)$.  Finally, $N$ is in $\Dcatf(R)$
by the long exact cohomology sequence.
\end{proof}

\begin{proposition}
\label{pro:AR_triangles_preserved}
An AR triangle in $\Dcatc(R)$ is also an AR triangle in $\Dcatf(R)$.
\end{proposition}

\begin{proof}
By \cite[lem.\ 4.3]{KrauseARZ}, each object in $\Dcatc(R)$ is a pure
injective object of $\Dcat(R)$.  Hence by \cite[prop.\
3.2]{KrauseARZ}, each AR triangle in $\Dcatc(R)$ is an AR triangle in
$\Dcat(R)$, and in particular in $\Dcatf(R)$.
\end{proof}

\begin{proposition}
\label{pro:AR2}
\begin{enumerate}

  \item  $\Dcatc(R)$ has right AR triangles if and only if
    ${}_{R}(\dual\!R)$ is in $\Dcatc(R)$. 

  \item  $\Dcatc(R)$ has left AR triangles if and only if $(\dual\!R)_R$
    is in $\Dcatc(R^{\opp})$.

\end{enumerate}
\end{proposition}

\begin{proof}
(1).  Suppose that $\Dcatc(R)$ has right AR triangles.  The object
${}_{R}R$ of $\Dcatc(R)$ has endomorphism ring $k$ which is local, so
there is an AR triangle $M \rightarrow N \rightarrow {}_{R}R
\rightarrow$ in $\Dcatc(R)$.  By Proposition
\ref{pro:AR_triangles_preserved}, it is even an AR triangle in
$\Dcatf(R)$.  On the other hand, Proposition \ref{pro:AR} gives that
there is also an AR triangle $\Sigma^{-1}({}_{R}(\dual\!R))
\rightarrow N^{\prime} \rightarrow {}_{R}R \rightarrow$ in
$\Dcatf(R)$, and since the right hand terms of the two AR triangles
are isomorphic, so are the left hand terms, $M \cong
\Sigma^{-1}({}_{R}(\dual\!R))$.  But $M$ is in $\Dcatc(R)$, so it
follows that $\Sigma^{-1}({}_{R}(\dual\!R))$ and hence
${}_{R}(\dual\!R)$ is in $\Dcatc(R)$.

Conversely, suppose that ${}_{R}(\dual\!R)$ is in $\Dcatc(R)$.  Given
$P$ in $\Dcatc(R)$, Proposition \ref{pro:AR} gives an AR triangle
\[
  \Sigma^{-1}(\dual\!R \LTensor_R P)
  \rightarrow N
  \rightarrow P
  \rightarrow
\]
in $\Dcatf(R)$.  Since ${}_{R}(\dual\!R)$ is in $\Dcatc(R)$, it is
finitely built from ${}_{R}R$.  The same is true for $P$, and so
$\dual\!R \LTensor_R P$ is also finitely built from ${}_{R}R$, that
is, it is in $\Dcatc(R)$.  It follows that both outer terms of the AR
triangle are in $\Dcatc(R)$, and then so is $N$.  That is, the AR
triangle is in $\Dcatc(R)$, so it is an AR triangle in that category.

(2).  The functors $\RHom_R(-,R)$ and $\RHom_{R^{\opp}}(-,R)$ are
quasi-inverse dualities between $\Dcatc(R)$ and $\Dcatc(R^{\opp})$,
so $\Dcatc(R)$ has left AR triangles if and only if
$\Dcatc(R^{\opp})$ has right AR triangles.  By the right module
version of part (1), this happens if and only if $(\dual\!R)_R$ is in
$\Dcatc(R^{\opp})$.
\end{proof}

\begin{theorem}
\label{thm:R}
The following conditions are equivalent.
\begin{enumerate}

  \item  $\Dcatc(R)$ has AR triangles.

  \item  $\Dcatc(R^{\opp})$ has AR triangles.

  \item  $R$ is Gorenstein.

\end{enumerate}
\end{theorem}

\begin{proof}
By Theorem \ref{thm:Gorenstein}(5), condition (3) is equivalent to
having that ${}_{R}(\dual\!R)$ is in $\Dcatc(R)$ and $(\dual\!R)_R$ is
in $\Dcatc(R^{\opp})$.  This is equivalent to condition (1) by
Proposition \ref{pro:AR2}, and it is equivalent to condition (2) by
the right module version of Proposition \ref{pro:AR2}.
\end{proof}

\begin{remark}
\label{rmk:tau}
Assume the situation of Theorem \ref{thm:R}.

Since $\Dcatc(R)$ has AR triangles, \cite[thm.\ 4.4]{KrauseARZ} and
Equation \eqref{equ:Serre} imply that
\begin{equation}
\label{equ:Serre2}
  S(-) = \dual\!R \LTensor_R -
\end{equation}
is a Serre functor of $\Dcatc(R)$, cf.\ Definition \ref{def:Serre}.
So the AR translation $\tau$ of $\Dcatc(R)$ extends to the
autoequivalence
\begin{equation}
\label{equ:m}
  \Sigma^{-1}(\dual\!R \LTensor_R -)
\end{equation}
of $\Dcatc(R)$, cf.\ Theorem \ref{thm:Serre}.  A quasi-inverse
equivalence is 
\[
  \Sigma \RHom_{R^{\opp}}(\dual\!R,R) \LTensor_R -;
\]
these two expressions can also be viewed as quasi-inverse
autoequivalences of $\Dcat(R)$.

If $X$ is an indecomposable object of $\Dcatc(R)$ then there are AR
triangles in $\Dcatc(R)$,
\[
  \Sigma^{-1}(\dual\!R \LTensor_R X) \rightarrow Y \rightarrow X \rightarrow
\]
and
\[
  X 
  \rightarrow Y^{\prime}
  \rightarrow \Sigma \RHom_{R^{\opp}}(\dual\!R,R) \LTensor_R X
  \rightarrow.
\]

Combining Equation \eqref{equ:m} with Theorem \ref{thm:Gorenstein}(3)
which says $(\dual\!R)_R \cong (\Sigma^d R)_R$ gives
\begin{equation}
\label{equ:j}
  \Homology(\tau(-)) \cong \Homology(\Sigma^{d-1}(-))
\end{equation}
as graded $k$-vector spaces.
\end{remark}

\section{The Auslander-Reiten quiver of a Differential Graded algebra:
  Local structure}
\label{sec:local}

This section considers the AR quiver $\Gamma$ of the compact derived
category $\Dcatc(R)$.  When $R$ is Gorenstein with $\dim_k
\Homology\!R \geq 2$, it is proved that each component of $\Gamma$ is
isomorphic to $\BZ A_{\infty}$ as a translation quiver.  The results
first appeared in \cite{arquiv}; the methods of Karsten Schmidt
\cite{Schmidt} have permitted some technical assumptions to be
removed.

\begin{setup}
\label{set:local}
In this section, $R$ will be Gorenstein with $\dim_k \Homology\!R \geq
2$.

The category $\Dcatc(R)$ has AR triangles by Theorem \ref{thm:R}, and
${}_{R}k$ is not in $\Dcatc(R)$ by Theorem \ref{thm:k}.

The AR quiver $\Gamma(\Dcatc(R))$ will be abbreviated to $\Gamma$.

Then $\Gamma$ with the AR translation $\tau$ is a stable translation
quiver by Proposition \ref{pro:stable_translation_quiver}.  By $C$
will be denoted a component of the translation quiver $\Gamma$.
\end{setup}

\begin{lemma}
\label{lem:no_loops}
\begin{enumerate}

  \item  No positive power $\tau^p$ of the AR translation $\tau$ has a
    fixed point in $\Gamma$.

  \item  $\Gamma$ has no loops.

\end{enumerate}
\end{lemma}

\begin{proof}
(1).  Remark \ref{rmk:tau} says $\tau(M) = \Sigma^{-1}(\dual\!R
\LTensor_R M)$.  Lemma \ref{lem:inf} implies
\[
  \inf \tau(M) = 1 + \inf \dual\!R + \inf M = 1 - d +\inf M.
\]
Since $d$ is either $0$ or $\geq 2$, it follows that each positive
power $\tau^p(M)$ has $\inf$ different from $\inf M$, so no positive
power is isomorphic to $M$.

(2).  The existence of a loop $[M] \rightarrow [M]$ would mean the
existence of an irreducible morphism $M \rightarrow M$ in
$\Dcatc(R)$.  Such a morphism would be in the radical of the finite
dimensional algebra $\Hom_{\Dcatc(R)}(M,M)$, and hence some power
would be zero.  Mimicking the proof of \cite[lem.\ VII.2.5]{ARS} now
shows $\tau(M) = M$, but this contradicts part (1).
\end{proof}

A reference for the graph theoretical terminology of the following
proposition is \cite[sec.\ 4.15]{BensonI}.  A salient fact is that
when $T$ is a directed tree, then the vertices of the repetitive
quiver $\BZ T$ have the form $(p,t)$ where $p$ is an integer, $t$ is a
vertex of $T$.  The translation of the stable translation quiver $\BZ
T$ is determined by $\tau(p,t) = (p+1,t)$.

\begin{proposition}
There exist a directed tree $T$ and an admissible group of
automorphisms $\Pi$ of $\BZ T$ so that $C \cong \BZ T/\Pi$ as stable
translation quivers.
\end{proposition}

\begin{proof}
Since $\tau$ extends to an autoequivalence of $\Dcatc(R)$ by Remark
\ref{rmk:tau}, the AR translation is an automorphism of $\Gamma$ so
restricts to an automorphism of $C$.  By definition, $C$ has no
multiple arrows, and by Lemma \ref{lem:no_loops}(2), it has no
loops.  Hence the proposition follows from the Riedtmann Structure
Theorem, see \cite[thm.\ 4.15.6]{BensonI}.
\end{proof}

To show that $T = A_{\infty}$ and that $\Pi$ acts trivially, the
following definitions are useful.

\begin{definition}
\label{def:varphi}
Define a function on the objects of $\Dcat(R)$ by
\[
  \varphi(M) = \dim_k \Ext_R(M,k).
\]
By abuse of notation, the induced function on the vertices of the AR
quiver $\Gamma$ is also denoted by $\varphi$.

Label the AR quiver $\Gamma$ by assigning to the arrow $[M]
\stackrel{\mu}{\rightarrow} [N]$ the label
$(\alpha_{\mu},\beta_{\mu})$, where $\alpha_{\mu}$ is the multiplicity
of $M$ as a direct summand of $Y$ in the AR triangle
\[
  \tau N \rightarrow Y \rightarrow N \rightarrow
\]
and $\beta_{\mu}$ is the multiplicity of $N$ as a direct summand of
$X$ in the AR triangle
\[
  M \rightarrow X \rightarrow \tau^{-1}M \rightarrow.
\]

The vertices of $\BZ T$ have the form $(p,t)$ where $p$ is an integer,
$t$ a vertex of $T$, so each vertex $t$ of $T$ gives a vertex $(0,t)$
of $\BZ T$ and hence a vertex $\Pi(0,t)$ of $\BZ T/\Pi$, that is, of
$C$.  Similarly, an arrow $t \rightarrow t^{\prime}$ in $T$ gives an
arrow $\Pi(0,t) \rightarrow \Pi(0,t^{\prime})$ of $C$.  Hence the
function $\varphi$ and the labelling $(\alpha,\beta)$ on
$\Gamma$ induce a function and a labelling on $T$.  These will be
denoted by $f$ and $(a,b)$.
\end{definition}

\begin{lemma}
\label{lem:varphi}
The function $\varphi$ and the labelling $(\alpha,\beta)$ have the
following properties.
\begin{enumerate}

  \item  If $F$ is a minimal semi-free resolution of $M$ with
    $F^{\natural} \cong \bigoplus_i \Sigma^i (R^{\natural})^{(\beta_i)}$,
    then $\varphi(M)$ is equal to the number of direct summands
    $\Sigma^i R^{\natural}$ in $F^{\natural}$.

  \item  $\varphi(\tau N) = \varphi(N)$.

  \item  If $\tau N \rightarrow Y \rightarrow N \rightarrow$ is an AR
    triangle in $\Dcatc(R)$, then $\varphi(Y) = \varphi(\tau N) +
    \varphi(N)$.

  \item If there is an arrow $[M] \stackrel{\mu}{\rightarrow} [N]$ in
  $\Gamma$ then there is a corresponding arrow $\tau[N]
  \stackrel{\nu}{\rightarrow} [M]$, and $(\alpha_{\nu},\beta_{\nu}) =
  (\beta_{\mu},\alpha_{\mu})$.
  
  \item If there is an arrow $[M] \stackrel{\mu}{\rightarrow} [N]$ in
  $\Gamma$ then there is also an arrow $\tau[M]
  \stackrel{\tau(\mu)}{\rightarrow} \tau[N]$, and
  $(\alpha_{\tau(\mu)},\beta_{\tau(\mu)}) =
  (\alpha_{\mu},\beta_{\mu})$.

  \item  $\sum_{\mu : [M] \rightarrow [N]}\alpha_{\mu}\varphi(M) =
    \varphi(\tau N) +  \varphi(N)$, where the sum is over all arrows
    in $\Gamma$ with target $[N]$.

\end{enumerate}
\end{lemma}

\begin{proof}
(1).  It holds that
\[
  \varphi(M)
  = \dim_k \Homology(\RHom_R(M,k))
  \stackrel{\rm (c)}{=} \dim_k \Homology(\Hom_R(F,k))
  \stackrel{\rm (d)}{=} \dim_k \Hom_{R^{\natural}}(F^{\natural},k^{\natural}),
\]
where (c) and (d) are by Lemma \ref{lem:semi-free}, parts (2) and (5).
The right hand side is clearly equal to the number of direct summands
$\Sigma^i R^{\natural}$ in $F^{\natural}$.

(2).  It holds that
\begin{align*}
  \varphi(\tau N)
  & \stackrel{\rm (a)}{=} \dim_k \Homology(\RHom_R(\Sigma^{-1}(\dual\!R \LTensor_R N),k)) \\
  & = \dim_k \Homology(\RHom_R(N,\Sigma \RHom_R(\dual\!R,k))) \\
  & \stackrel{\rm (b)}{=} \dim_k \Homology(\RHom_R(N,\Sigma^{1-d}k)) \\
  & = \dim_k \Homology(\RHom_R(N,k)) \\
  & = \varphi(N),
\end{align*}
where (a) is by Remark \ref{rmk:tau} and (b) follows from Theorem
\ref{thm:Gorenstein}(3).

(3).  The AR triangle of the lemma induces a long exact sequence
consisting of pieces
\[
  \Ext_R^i(N,k)
  \rightarrow \Ext_R^i(Y,k)
  \rightarrow \Ext_R^i(\tau N,k),
\]
and the claim will follow if the connecting maps are zero.

Indeed, the AR triangle is also an AR triangle in $\Dcatf(R)$ by
Proposition \ref{pro:AR_triangles_preserved}.  A morphism $\tau N
\rightarrow {}_{R}(\Sigma^i k)$ in $\Dcatf(R)$ cannot be a split
monomorphism since $\tau N$ is in $\Dcatc(R)$ while ${}_{R}(\Sigma^i
k)$ is not, cf.\ Setup \ref{set:local}.  It follows that each such
morphism factors through $\tau N \rightarrow Y$ whence the composition
$\Sigma^{-1}N \rightarrow \tau N \rightarrow \Sigma^i k$ is zero.
Hence the connecting morphism $\Ext_R^i(\tau N,k) \rightarrow
\Ext_R^{i+1}(N,k)$ is zero as desired.

(4).  Let
\begin{equation}
\label{equ:e}
  \tau N \rightarrow Y \rightarrow N \rightarrow
\end{equation}
be an AR triangle in $\Dcatc(R)$.  By the definition of the labelling
of $\Gamma$, the multiplicity of $M$ as a direct summand of $Y$ is
equal to both $\beta_{\nu}$ and $\alpha_{\mu}$, so $\beta_{\nu} =
\alpha_{\mu}$.  A similar argument shows $\alpha_{\nu} = \beta_{\mu}$,
so $(\alpha_{\nu},\beta_{\nu}) = (\beta_{\mu},\alpha_{\mu})$.

(5).  This holds since the AR translation $\tau$ of $\Dcatc(R)$ is the
restriction of an equivalence of categories by Remark \ref{rmk:tau}.

(6).  Consider the AR triangle \eqref{equ:e}.  The object
$Y$ is a direct sum of copies of the indecomposable objects of
$\Dcatc(R)$ which have irreducible morphisms to $N$, and the
multiplicity of $M$ as a direct summand of $Y$ is $\alpha_{\mu}$ where
$[M] \stackrel{\mu}{\rightarrow} [N]$ is the arrow in $\Gamma$.  Hence
\[
  \sum_{\mu:[M] \rightarrow [N]} \alpha_{\mu}\varphi(M) = \varphi(Y).
\]
Now combine with part (3).
\end{proof}

\begin{lemma}
\label{lem:zero}
Let $M \langle 0 \rangle, \ldots, M \langle 2^p - 1 \rangle$ be
indecomposable objects of $\Dcatc(R)$ with $\varphi(M \langle i
\rangle) \leq \frac{p}{\dim_k R}$ for each $i$.  If 
\[
  M \langle 2^p - 1 \rangle \rightarrow M \langle 2^p - 2 \rangle
  \rightarrow \cdots \rightarrow M \langle 0 \rangle
\]
are non-isomorphisms in $\Dcatc(R)$, then the composition is zero.
\end{lemma}

\begin{proof}
Let $F \langle i \rangle$ be a minimal semi-free resolution of $M
\langle i \rangle$.  Each $F \langle i \rangle$ must be
indecomposable as a DG left-$R$-module, for if $F \langle i \rangle$
decomposed then it would do so into DG modules $F
\langle i_{\alpha} \rangle$ with $\partial(F \langle i_{\alpha}
\rangle) \subseteq R^{\geq 1} \cdot F \langle i_{\alpha} \rangle$, but
this condition forces non-zero cohomology so the decomposition of $F
\langle i \rangle$ as a DG module would induce a non-trivial
decomposition of $M \langle i \rangle$ in $\Dcatc(R)$.

The morphisms in $\Dcatc(R)$ between the $M \langle i \rangle$ are
represented by morphisms
\begin{equation}
\label{equ:f}
  F \langle 2^p - 1 \rangle \rightarrow F \langle 2^p - 2 \rangle
  \rightarrow \cdots \rightarrow F \langle 0 \rangle
\end{equation}
of DG left-$R$-modules.  These cannot be bijections, since if they
were, then the morphisms in $\Dcatc(R)$ between the $M \langle i
\rangle$ would be isomorphisms.

Now note that if $F \langle i \rangle^{\natural} = \bigoplus_j
\Sigma^j(R^{\natural})^{(\beta_j)}$, then the direct sum has
$\varphi(M \langle i \rangle)$ summands $\Sigma^j R^{\natural}$ by
Lemma \ref{lem:varphi}(1).  Hence
\[
  \dim_k F \langle i \rangle
  = \varphi(M \langle i \rangle) \dim_k R
  \leq \frac{p}{\dim_k R} \dim_k R
  = p,
\]
and it is not hard to mimick the proof of \cite[lem.\ 4.14.1]{BensonI}
to see that hence, the composition of the morphisms in Equation
\eqref{equ:f} is zero.  This implies that the composition of the
morphisms in the lemma is zero.
\end{proof}

\begin{lemma}
\label{lem:non-zero}
If $M \langle 0 \rangle$ is an indecomposable object of $\Dcatc(R)$
and $q \geq 0$ is an integer, then there exist indecomposable objects
and irreducible morphisms in $\Dcatc(R)$,
\[
  M \langle q \rangle \rightarrow M \langle q-1 \rangle
  \rightarrow \cdots \rightarrow M \langle 0 \rangle,
\]
with non-zero composition.
\end{lemma}

\begin{proof}
Let me prove a stronger statement which implies the lemma: If $M
\langle 0 \rangle$ is an indecomposable object of $\Dcatc(R)$ and $q
\geq 0$ is an integer, then there exists
\[
  {}_{R}(\Sigma^i k)
  \stackrel{\kappa_q}{\rightarrow} M \langle q \rangle
  \stackrel{\mu_q}{\rightarrow} M \langle q-1 \rangle
  \stackrel{\mu_{q-1}}{\rightarrow} \cdots
  \stackrel{\mu_1}{\rightarrow} M \langle 0 \rangle
\]
where the $M \langle i \rangle$ are indecomposable objects of
$\Dcatc(R)$ and the $\mu_i$ are irreducible morphisms in $\Dcatc(R)$,
such that $\mu_1 \circ \cdots \circ \mu_q \circ \kappa_q \neq 0$.

Using induction on $q$, first let $q = 0$.  Let $F$ be a minimal
semi-free resolution of the dual $\dual\!M \langle 0 \rangle$.  Then
\begin{align*}
  \Homology(\RHom_R(k,M \langle 0 \rangle))
  & \cong \Homology(\RHom_{R^{\opp}}(\dual\!M \langle 0 \rangle,k)) \\
  & \stackrel{\rm (a)}{\cong} \Homology(\Hom_{R^{\opp}}(F,k)) \\
  & \stackrel{\rm (b)}{\cong} \Hom_{(R^{\opp})^{\natural}}(F^{\natural},k^{\natural}) \\
  & \stackrel{\rm (c)}{\not\cong} 0.
\end{align*}
Here (a) and (b) are by Lemma \ref{lem:semi-free}, parts (2) and (5).
(c) is because $M \langle 0 \rangle$ is indecomposable hence has
non-zero cohomology; this implies that $\dual\!M \langle 0 \rangle$
has non-zero cohomology, and then $F$ is non-trivial semi-free whence
$F^{\natural}$ is a non-trivial graded free module.

It follows from the displayed formula that there is a non-zero
morphism
\[
  {}_{R}(\Sigma^i k) \stackrel{\kappa_0}{\rightarrow} M \langle 0 \rangle
\]
for some $i$.

Now let $q \geq 1$ and suppose that
\[
  {}_{R}(\Sigma^i k)
  \stackrel{\kappa_{q-1}}{\rightarrow} M \langle q-1 \rangle
  \stackrel{\mu_{q-1}}{\rightarrow} M \langle q-2 \rangle
  \stackrel{\mu_{q-2}}{\rightarrow} \cdots
  \stackrel{\mu_1}{\rightarrow} M \langle 0 \rangle
\]
has already been found with the desired properties.  Let $\tau M
\langle q-1 \rangle \rightarrow X \langle q \rangle
\stackrel{\mu_{q}^{\prime}}{\rightarrow} M \langle q-1 \rangle
\rightarrow$ be an AR triangle in $\Dcatc(R)$.  By Proposition
\ref{pro:AR_triangles_preserved} it is also an AR triangle in
$\Dcatf(R)$.  Since ${}_{R}k$ is not in $\Dcatc(R)$, see Setup
\ref{set:local}, it is clear that $\kappa_{q-1}$ is not a split
epimorphism, so it factors through $\mu_{q}^{\prime}$.  Now I can get
the situation claimed in the lemma by letting $M \langle q \rangle$ be
a suitable indecomposable summand of $X \langle q \rangle$ and $\mu_q$
the restriction of $\mu^{\prime}_q$ to $M \langle q \rangle$.
\end{proof}

\begin{lemma}
\label{lem:varphi_unbounded}
The function $\varphi$ is unbounded on $C$. 
\end{lemma}

\begin{proof}
If $\varphi$ were bounded on $C$ then Lemma \ref{lem:zero} would apply
to sufficiently long sequences of morphisms between indecomposable
objects with vertices in $C$, but this would make impossible the
situation established in Lemma \ref{lem:non-zero}.
\end{proof}

Recall that the Cartan matrix $c$ of the labelled directed tree $T$ is
a matrix with rows and columns indexed by the vertices of $T$.  If $s$
and $t$ are vertices, then
\[
  c_{st} =
  \left\{
    \begin{array}{cl}
      2 & \mbox{if $s=t$}, \\
      -a_{\mu} & \mbox{if there is an arrow $s \stackrel{\mu}{\rightarrow} t$}, \\
      -b_{\nu} & \mbox{if there is an arrow $t \stackrel{\nu}{\rightarrow} s$}, \\
      0 & \mbox{if $s \neq t$ and $s$ and $t$ are not connected by an arrow};
    \end{array}
  \right.
\]
cp.\ \cite[sec.\ 4.5]{BensonI}.
The function $f$ on the vertices of $T$ is called additive if it
satisfies $\sum_s c_{st}f(s) = 0$ for each $t$, that is,
\begin{equation}
\label{equ:g}
  2f(t) - \sum_{\mu : s \rightarrow t}a_{\mu}f(s)
        - \sum_{\nu : t \rightarrow u}b_{\nu}f(u) = 0
\end{equation}
for each $t$, where the sums are over all arrows in $T$ into $t$ and
out of $t$.  Indeed:

\begin{proposition}
\label{pro:f}
The function $f$ is additive and unbounded on $T$.
\end{proposition}

\begin{proof}
Using Definition \ref{def:varphi}, the left hand side of Equation
\eqref{equ:g} can be rewritten
\[
  2\varphi(\Pi(0,t))
  - \sum_{\mu : s \rightarrow t} \alpha_{\Pi(0,s) \rightarrow \Pi(0,t)} \varphi(\Pi(0,s))
  - \sum_{\nu : t \rightarrow u} \beta_{\Pi(0,t) \rightarrow \Pi(0,u)} \varphi(\Pi(0,u)).
\]
The translation of $\BZ T/\Pi$ is given by $\tau(\Pi(p,t)) =
\Pi(p+1,t)$.  To each arrow $\Pi(0,t) \rightarrow \Pi(0,u)$
corresponds an arrow $\tau(\Pi(0,u)) \rightarrow \Pi(0,t)$, that is,
$\Pi(1,u) \rightarrow \Pi(0,t)$.  Lemma \ref{lem:varphi}(4) gives
$\beta_{\Pi(0,t) \rightarrow \Pi(0,u)} = \alpha_{\Pi(1,u) \rightarrow
\Pi(0,t)}$.  Lemma \ref{lem:varphi}(2) gives $\varphi(\Pi(0,u)) =
\varphi(\tau(\Pi(0,u))) = \varphi(\Pi(1,u))$, and also implies that
$2\varphi(\Pi(0,t)) = \varphi(\Pi(0,t)) + \varphi(\tau(\Pi(0,t)))$.

Substituting all this into the previous expression gives
\begin{align*}
  \lefteqn{\varphi(\Pi(0,t)) + \varphi(\tau(\Pi(0,t)))} & \\
  & \hspace{10ex} - \sum_{\mu : s \rightarrow t} \alpha_{\Pi(0,s) \rightarrow \Pi(0,t)} \varphi(\Pi(0,s))
  - \sum_{\nu : t \rightarrow u} \alpha_{\Pi(1,u) \rightarrow \Pi(0,t)} \varphi(\Pi(1,u)).
\end{align*}
Recall that the sums are over all the arrows in $T$ into $t$ and out
of $t$.  From the construction of the repetitive quiver $\BZ T$, this
means that between them, the sums can be viewed as being over all the
arrows into $(0,t)$ in $\BZ T$.  However, the projection $\BZ T
\rightarrow \BZ T/\Pi$ is a covering so induces a bijection between
the arrows in $\BZ T$ into $(0,t)$ and the arrows in $\BZ T/\Pi$ into
$\Pi(0,t)$.  So in fact, the previous expression can be rewritten
\[
  \varphi(\Pi(0,t)) + \varphi(\tau(\Pi(0,t)))
  - \sum_{m \rightarrow \Pi(0,t)} \alpha_{m \rightarrow \Pi(0,t)} \varphi(m)
\]
where the sum is over all arrows in $\BZ T/\Pi$ into $\Pi(0,t)$.  But
identifying $\BZ T/\Pi$ and $C$, the displayed expression is zero by
Lemma \ref{lem:varphi}(6), so $f$ is additive.

Since $f(t) = \varphi(\Pi(0,t))$ by Definition \ref{def:varphi} and
$\varphi(\Pi(p,t)) = \varphi(\tau^p\Pi(0,t)) = \varphi(\Pi(0,t))$ by
Lemma \ref{lem:varphi}(2), if $f$ were bounded on $T$ then $\varphi$
would be bounded on $C$.  But this is false by Lemma
\ref{lem:varphi_unbounded}.
\end{proof}

Recall that the graph $A_{\infty}$ is
\[
  \def\objectstyle{\scriptstyle}
  \vcenter{
  \xymatrix @!0 @+0.25pc {
    1 \ar@{-}[rr] & & 2 \ar@{-}[rr] & & 3 \ar@{-}[rr] & & 4 \ar@{-}[rr] & & 5 \ar@{-}[rr]&&{\textstyle \cdots}\\
                         }
          },
\]
where a convenient numbering of the vertices has been chosen.  A
quiver of type $A_{\infty}$ is an orientation of this graph.  The
repetitive quiver $\BZ A_{\infty}$ does not depend on the orientation;
with a standard numbering of the vertices it is
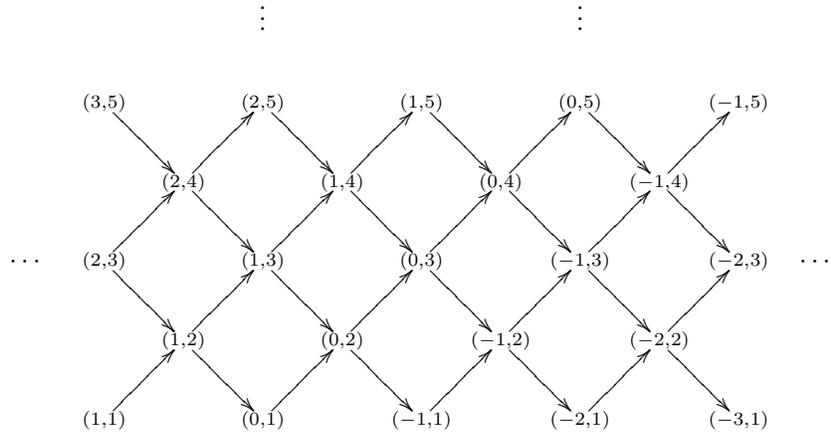
\[
  \def\objectstyle{\scriptstyle}
  \vcenter{
  \xymatrix @!0 @+0.5pc {
    & & & \vdots & & & & \vdots & & & \\
    & *{(3,5)} \ar[dr] & & *{(2,5)} \ar[dr] & & *{(1,5)} \ar[dr] & & *{(0,5)} \ar[dr] & & *{(-1,5)} & \\
    & & *{(2,4)} \ar[dr] \ar[ur] & & *{(1,4)} \ar[dr] \ar[ur] & & *{(0,4)} \ar[dr] \ar[ur] & & *{(-1,4)} \ar[dr] \ar[ur] & & \\
    {\textstyle \cdots} & *{(2,3)} \ar[dr] \ar[ur] & *{} & *{(1,3)} \ar[dr] \ar[ur] & *{} & *{(0,3)} \ar[dr] \ar[ur] & *{} & *{(-1,3)} \ar[dr] \ar[ur] & *{} & *{(-2,3)} & {\textstyle \cdots}.\\
    & & *{(1,2)} \ar[dr] \ar[ur] & & *{(0,2)} \ar[dr] \ar[ur] & & *{(-1,2)} \ar[dr] \ar[ur] & & *{(-2,2)} \ar[dr] \ar[ur] & & \\
    & *{(1,1)} \ar[ur] & & *{(0,1)} \ar[ur] & & *{(-1,1)} \ar[ur] & & *{(-2,1)} \ar[ur] & & *{(-3,1)} & \\
                      }
          }
\]
The translation acts by $\tau(p,t) = (p+1,t)$.

\begin{theorem}
\label{thm:quiver_local_structure}
\begin{enumerate}
  \item  The component $C$ of the AR quiver $\Gamma$ of $\Dcatc(R)$ is
    isomorphic to $\BZ A_{\infty}$ as a stable translation quiver.

  \item  Each label $(\alpha_{\mu},\beta_{\mu})$ on $\Gamma$ is equal
    to $(1,1)$. 

  \item  If the function $\varphi$ has value $\varphi_1$ on the edge of
    $C \cong \BZ A_{\infty}$, then it has value $n\varphi_1$ on the
    $n$'th horizontal row of vertices in $C$.
\end{enumerate}
\end{theorem}

\begin{proof}
By Proposition \ref{pro:f}, there is an additive unbounded function
$f$ on the labelled tree $T$.  Hence $T$ is of type $A_{\infty}$ with
all labels equal to $(1,1)$ by \cite[thm.\ 4.5.8(iv)]{BensonI}.  This
proves (2), and it also means that to prove (1), it is sufficient to
show that $\Pi$ acts trivially on $\BZ A_{\infty}$.

But if it did not, then there would exist a vertex $m$ on the edge of
$\BZ A_{\infty}$ and a $g$ in $\Pi$ such that $gm \neq m$.  The vertex
$gm$ would again be on the edge, and so it would have the form $\tau^p
m$ for some $p \neq 0$.  But then $m$ and $\tau^p m$ would get
identified in $\BZ A_{\infty}/\Pi$, and hence $\Pi m$ would be a fixed
point in $\BZ A_{\infty}/\Pi$ of $\tau^p$, that is, a fixed point in
$C$ of $\tau^p$.  But this is impossible by Lemma
\ref{lem:no_loops}(1).

Finally, it is a standard consequence of additivity that if the
function $f$ has value $f(1) = f_1$ at the first vertex of
$A_{\infty}$, then it has value $f(n) = nf_1$ at the $n$th vertex.
Since $\varphi(\Pi(p,n)) = \varphi(\tau^p(\Pi(0,n))) =
\varphi(\Pi(0,n)) = f(n)$, the claim (3) on $\varphi$ follows.
\end{proof}

\section{Report on work by Karsten Schmidt}
\label{sec:global}

In this section, the study of the AR quiver $\Gamma$ of $\Dcatc(R)$ is
continued, and some aspects of the global structure are revealed.  If
$\dim_k \Homology\!R = 2$ then $\Gamma$ has precisely $d-1$
components.  On the other hand, for Gorenstein algebras with $\dim_k
\Homology\!R \geq 3$, there are infinitely many components.  Often,
these even form families which are indexed by projective manifolds,
and these manifolds can be of arbitrarily high dimension.

With the exception of Theorem \ref{thm:spheres} which is essentially
in \cite{artop}, the results of this section are due to Karsten
Schmidt; see \cite[thm.\ 4.1]{Schmidt}.

Only a sketch is given of the proof of the next theorem; for more
information, see \cite[sec.\ 8]{artop}.

\begin{theorem}
\label{thm:spheres}
If $\dim_k \Homology\!R = 2$ then $\Dcatc(R)$ has AR triangles, and
the AR quiver of $\Dcatc(R)$ has $d-1$ components, each isomorphic to
$\BZ A_{\infty}$.
\end{theorem}

\begin{proof}
The cohomology of $R$ in low degrees is $\Homology^0\!R = k$ and
$\Homology^1\!R = 0$.  Since $\dim_k \Homology\!R = 2$, it follows
that the only other non-zero cohomology is $\Homology^d\!R = k$, and
it is easy to check that $R$ therefore satisfies the conditions of
Theorem \ref{thm:Gorenstein}(2) so $R$ is Gorenstein.  Theorem
\ref{thm:R} says that $\Dcatc(R)$ has AR triangles, and Theorem
\ref{thm:quiver_local_structure}(1) says that each component of the AR
quiver of $\Dcatc(R)$ is isomorphic to $\BZ A_{\infty}$.

Replacing $R$ with a quasi-isomorphic truncation, it can be supposed
that $R^{>d} = 0$, see Lemma \ref{lem:truncations}(3).  Pick a cycle
$x$ in $R^d$ with non-zero cohomology class.  The graded algebra
$k[X]/(X^2)$ with $X$ in cohomological degree $d$ can be viewed as a
DG algebra with zero differential, and the map $k[X]/(X^2) \rightarrow
R$ sending $X$ to $x$ is a quasi-isomorphism, so $R$ can be replaced
with $k[X]/(X^2)$.

Now consider the algebra $S = k[Y]$ with $Y$ in cohomological degree
$-d + 1$, viewed as a DG algebra with zero differential.  The DG
module $k$ can be viewed as a DG right-$R$-right-$S$-module in an
obvious way, and it induces adjoint functors
\[
  \xymatrix{
    \Dcat(S^{\opp}) \ar@<-1ex>[rr]_{\RHom_{S^{\opp}}(k,-)}
    & & \Dcat(R). \ar@<-1ex>[ll]_{k \LTensor_R -}
           }
\]
The upper functor clearly sends ${}_{R}R$ to $k_S$, and by computing
with a semi-free resolution it can be verified that the lower functor
sends $k_S$ to ${}_{R}R$.  Hence the functors restrict to
quasi-inverse equivalences on the subcategories of objects which are
finitely built, respectively, from $k_S$ and ${}_{R}R$.  These
subcategories are precisely $\Dcatf(S^{\opp})$ and $\Dcatc(R)$.

So it is enough to show that the AR quiver of $\Dcatf(S^{\opp})$ has
$d-1$ components.  However, $S$ is $k[Y]$ equipped with zero
differential, so $\Homology\!S$ is just $k[Y]$ viewed as a graded
algebra.  This polynomial algebra in one variable has global dimension
$1$, and this makes it possible to prove that if $M$ is a DG
right-$S$-module, then $M$ is quasi-isomorphic to $\Homology\!M$
equipped with zero differential.

This reduces the classification of objects of $\Dcatf(S^{\opp})$ to
the classification of graded right-$\Homology\!S$-modules.  However,
using again that $\Homology\!S = k[Y]$ is a polynomial algebra in one
variable, one shows that its indecomposable finite dimensional graded
right-modules are precisely
\[
  \Sigma^j k[Y]/(Y^{m+1})
\]
for $j$ in $\BZ$ and $m \geq 0$.  Viewing these as DG
right-$S$-modules with zero differential gives the indecomposable
objects of $\Dcatf(S^{\opp})$, and knowing the indecomposable objects,
it is an exercise in AR theory to compute the AR triangles, find the
AR quiver, and verify that it has $d-1$ components.
\end{proof}

\begin{setup}
In the rest of this section, the setup of Section \ref{sec:local} will
be kept: $R$ is Gorenstein with $\dim_k \Homology\!R \geq 2$.

The category $\Dcatc(R)$ has AR triangles by Theorem \ref{thm:R}, and
${}_{R}k$ is not in $\Dcatc(R)$ by Theorem \ref{thm:k}.

The AR quiver $\Gamma(\Dcatc(R))$ will abbreviated to $\Gamma$.
\end{setup}

Since $\Homology^0\!R \cong k$ and $\Homology^1\!R = 0$, by Theorem
\ref{thm:Gorenstein}(2) it must be the case that $\Homology^{d-1}\!R =
0$ and $\Homology^d\!R \cong k$.  By definition, $d$ is the highest
degree in which $R$ has non-zero cohomology; suppose that $e \not\in \{
0,d \}$ is another degree with $\Homology^e\!R \neq 0$ and observe
that then
\[
  2 \leq e \leq d-2
\]
and $d \geq 4$.

Let $X$ be a minimal semi-free DG left-$R$-module whose semi-free
filtration contains only finitely many copies of (de)suspensions of
$R$.  In particular, Lemma \ref{lem:semi-free}(4) says that $X$ is in
$\Dcatc(R)$; suppose that it is indecomposable in that category.  Let
$i \geq 2$ and consider the following cases.

\begin{description}

  \item[Case (1).] Suppose that
\[
  \mbox{$\inf X = 0 \;\;$ and $\;\; \sup X = i$.}
\]
A non-zero cohomology class in $\Homology^i\!X$ permits a
non-zero morphism $\Sigma^{-i}R \stackrel{g}{\rightarrow} X$;
denoting the mapping cone by $X(1)$, there is a distinguished triangle
\begin{equation}
\label{equ:1}
  \Sigma^{-i}R \stackrel{g}{\rightarrow} X \rightarrow X(1) \rightarrow.
\end{equation}

  \item[Case (2).]  Suppose that
\[
  \mbox{$\inf X = 0, \;\; \sup X = i, \;\;$
        and $\Homology^{i-d+e}X \neq 0$.}
\]
A non-zero cohomology class in $\Homology^{i-d+e}\!X$ permits a
non-zero morphism $\Sigma^{-i+d-e}R \stackrel{h}{\rightarrow} X$;
denoting the mapping cone by $X(2)$, there is a distinguished triangle
\begin{equation}
\label{equ:2}
  \Sigma^{-i+d-e}R \stackrel{h}{\rightarrow} X \rightarrow X(2) \rightarrow.
\end{equation}

  \item[Case (2${}_{\alpha}$).]  In Case (2), suppose moreover
  that $\Homology^i\!X \cong k$ and that scalar multiplication induces
  a non-degenerate bilinear form
\begin{equation}
\label{equ:b}
  \Homology^{d-e}(R) \times \Homology^{i-d+e}(X)
  \rightarrow \Homology^i(X) \cong k.
\end{equation}

The morphism $\Sigma^{-i+d-e}R \stackrel{h}{\rightarrow} X$
corresponds to an element $\alpha$ in $\Homology^{i-d+e}\!X$; denote
$h$ by $h_{\alpha}$ and $X(2)$ by $X(2_{\alpha})$.

\end{description}

It follows from the mapping cone construction that $X(1)$, $X(2)$, and
$X(2_{\alpha})$ are again minimal semi-free DG left-$R$-modules whose
semi-free filtrations contain only fi\-ni\-te\-ly many copies of
(de)suspensions of $R$.

\begin{lemma}
\label{lem:C}
\begin{enumerate}

  \item In Case (1) of the above construction, the DG module $X(1)$ is
  indecomposable in $\Dcatc(R)$.  It has
\[
  \inf X(1) = 0, \;\; \sup X(1) = i+d-1
\]
and
\[
  \Homology^{i+e-1}(X(1)) \cong \Homology^e(R) \neq 0, \;\;
  \Homology^{i+d-1}(X(1)) \cong \Homology^d(R) \cong k.
\]
It satisfies $\amp(X(1)) = \amp(X) + d - 1$ and $\varphi(X(1)) =
\varphi(X) + 1$.  Moreover, if the construction is applied to $X$ and
$X^{\prime}$ then $X(1) \cong X^{\prime}(1)$ implies $X \cong
X^{\prime}$ in $\Dcatc(R)$.  Finally, scalar multiplication induces a
non-degenerate bilinear form
\[
  \Homology^{d-e}(R) \times \Homology^{i+e-1}(X(1))
  \rightarrow \Homology^{i+d-1}(X(1)) \cong k.
\]

  \item In Case (2), the DG module $X(2)$ is indecomposable in
  $\Dcatc(R)$.  It has
\[
  \mbox{$\inf X(2) = 0 \;\;$ and $\;\; \sup X(2) = i+e-1$}.
\]
It satisfies $\amp(X(2)) = \amp(X) + e - 1$ and $\varphi(X(2)) =
\varphi(X) + 1$.  Moreover, if the construction is applied to $X$ and
$X^{\prime}$ then $X(2) \cong X^{\prime}(2)$ implies $X \cong
X^{\prime}$ in $\Dcatc(R)$.

  \item In Case (2${}_{\alpha}$), if $\alpha$ and $\alpha^{\prime}$ are
  elements of $\Homology^{i-d+e}\!X$ then
\[
  \mbox{$X(2_{\alpha}) \cong X(2_{\alpha^{\prime}})$ in $\Dcatc(R)$
    \; $\Leftrightarrow$ \;
    $\alpha = \kappa\alpha^{\prime}$ for a $\kappa$ in $k$.} 
\]
\end{enumerate}
\end{lemma}

\begin{proof}
(1).  Indecomposability will follow from \cite[lem.\ 6.5]{HKR} if
I can show in $\Dcatc(R)$ that $g$ is non-zero (clear), non-invertible
(clear since $\inf \Sigma^{-i}R = i \geq 2$ but $\inf X = 0$), and
that $\Hom_{\Dcatc(R)}(X,\Sigma\Sigma^{-i}R) = 0$.  However,
\begin{align*}
  \Hom_{\Dcatc(R)}(X,\Sigma\Sigma^{-i}R)
  & \cong \Hom_{\Dcatc(R^{\opp})}(\dual\!\Sigma\Sigma^{-i}R,\dual\!X) \\
  & \cong \Hom_{\Dcatc(R^{\opp})}(\Sigma^{i-1}\dual\!R,\dual\!X) \\
  & \stackrel{\rm (a)}{\cong} \Hom_{\Dcatc(R^{\opp})}(\Sigma^{i-1+d}R,\dual\!X) \\
  & \cong \Homology^{-i+1-d}(\dual\!X) \\
  & \cong \dual\!\Homology^{i-1+d}(X) \\
  & \stackrel{\rm (b)}{=} 0
\end{align*}
where (a) is by Theorem \ref{thm:Gorenstein}(3) and (b) is because
$\sup X = i$.

The statements $\inf X(1) = 0$, $\sup X(1) = i+d-1$,
$\Homology^{i+e-1}(X(1)) \cong \Homology^e(R) \neq 0$, and
$\Homology^{i+d-1}(X(1)) \cong \Homology^d(R) \cong k$ follow from the
long exact cohomology sequence of the distinguished triangle
\eqref{equ:1}.  The statement about the amplitude is a consequence,
and $\varphi(X(1)) = \varphi(X) + 1$ because $X(1)$ is minimal
semi-free with one more copy of a desuspension of $R$ in its
semi-free filtration than $X$; cf. Lemma \ref{lem:varphi}(1).

To get the statement on isomorphisms, first observe that by a
computation like the one above,
\[
  \Hom_{\Dcatc(R)}(X(1),\Sigma\Sigma^{-i}R)
  \cong \dual\!\Homology^{i+d-1}(X(1))
  \cong \dual(k)
  \cong k.
\]
Now suppose that there is an isomorphism $X(1)
\stackrel{\sim}{\rightarrow} X^{\prime}(1)$ in $\Dcatc(R)$.  This
gives a diagram between the distinguished triangles defining $X(1)$
and $X^{\prime}(1)$,
\[
  \xymatrix{
  \Sigma^{-i}R \ar[r] & X \ar[r] & X(1) \ar[r] \ar[d] & \Sigma^{-i+1}R \\
  \Sigma^{-i}R \ar[r] & X^{\prime} \ar[r] & X^{\prime}(1) \ar[r] & \Sigma^{-i+1}R \lefteqn{.}
           }
\]
The last morphism in the upper distinguished triangle is non-zero, for
otherwise the triangle would be split contradicting that $X$ is
indecomposable.  Since $\Hom_{\Dcatc(R)}(X(1),\Sigma\Sigma^{-i}R)$ is
one-dimensional, there exists a morphism $\Sigma^{-i+1}R \rightarrow
\Sigma^{-i+1}R$ to give a commutative square.  Adding this morphism
and its desuspension to the diagram gives
\[
  \xymatrix{
  \Sigma^{-i}R \ar[r] \ar[d] & X \ar[r] & X(1) \ar[r] \ar[d] & \Sigma^{-i+1}R \ar[d]\\
  \Sigma^{-i}R \ar[r] & X^{\prime} \ar[r] & X^{\prime}(1) \ar[r] & \Sigma^{-i+1}R \lefteqn{,}
           }
\]
and the two new vertical arrows are also isomorphisms since they are
non-zero and since $\Hom_{\Dcatc(R)}(R,R) \cong k$.  By the axioms of
triangulated categories, there is a vertical morphism $X \rightarrow
X^{\prime}$ which completes to a commutative diagram, and this
morphism is an isomorphism by the triangulated five lemma.

Finally, to get the non-degenerate bilinear form, observe that $R$ is
Gorenstein so by Theorem \ref{thm:Gorenstein}(2) scalar multiplication
gives a non-degenerate bilinear form
\[
  \Homology^{d-e}(R) \times \Homology^{i+e-1}(\Sigma^{-i+1}R)
  \rightarrow \Homology^{i+d-1}(\Sigma^{-i+1}R) \cong k.
\]
But $X(1)$ is a mapping cone which in degrees $\geq i+1$ is equal to
$\Sigma^{-i+1}R$, so this gives a non-degenerate bilinear form
\[
  \Homology^{d-e}(R) \times \Homology^{i+e-1}(X(1))
  \rightarrow \Homology^{i+d-1}(X(1)) \cong k
\]
as claimed.

(2)  follows by similar arguments.

(3).  $\Leftarrow$ is elementary.  $\Rightarrow$:  Given the isomorphism
$X(2_{\alpha}) \rightarrow X(2_{\alpha^{\prime}})$, the method applied
in the proof of (1) produces a diagram between the distinguished
triangles defining $X(2_{\alpha})$ and $X(2_{\alpha^{\prime}})$,
\[
  \xymatrix{
  \Sigma^{-i+d-e}R \ar[r]^-{h_{\alpha}} \ar[d] & X \ar[r] \ar[d]^{\gamma} & X(2_{\alpha}) \ar[r] \ar[d] & \Sigma^{-i+d-e+1}R \ar[d]\\
  \Sigma^{-i+d-e}R \ar[r]_-{h_{\alpha^{\prime}}} & X \ar[r] & X(2_{\alpha^{\prime}}) \ar[r] & \Sigma^{-i+d-e+1}R \lefteqn{,}
           }
\]
where the vertical maps are isomorphisms.  Commutativity of the first
square implies $(\Homology^{i-d+e}(\gamma))(\alpha) = \alpha^{\prime}$.

Consider $x$ in $\Homology^{d-e}\!R$.  Then
\[
  x\alpha^{\prime} = 0
  \Leftrightarrow x(\Homology^{i-d+e}(\gamma))(\alpha) = 0
  \Leftrightarrow (\Homology^{i-d+e}(\gamma))(x\alpha) = 0
  \Leftrightarrow x\alpha = 0,
\]
the last $\Leftrightarrow$ because $\gamma$ is an isomorphism in
$\Dcatc(R)$ whence $\Homology^{i-d+e}(\gamma)$ is bijective.  Seeing
that the bilinear form \eqref{equ:b} is non-degenerate, this means
that $\alpha = \kappa\alpha^{\prime}$ for a $\kappa$ in $k$.
\end{proof}

Observe that it makes sense to insert $X(1)$ into either of Cases (1),
(2), and (2${}_{\alpha}$).  Likewise, it makes sense to insert $X(2)$
and $X(2_{\alpha})$ into Case (1).  Iterating Cases (1) and (2), the
following tree can be constructed.
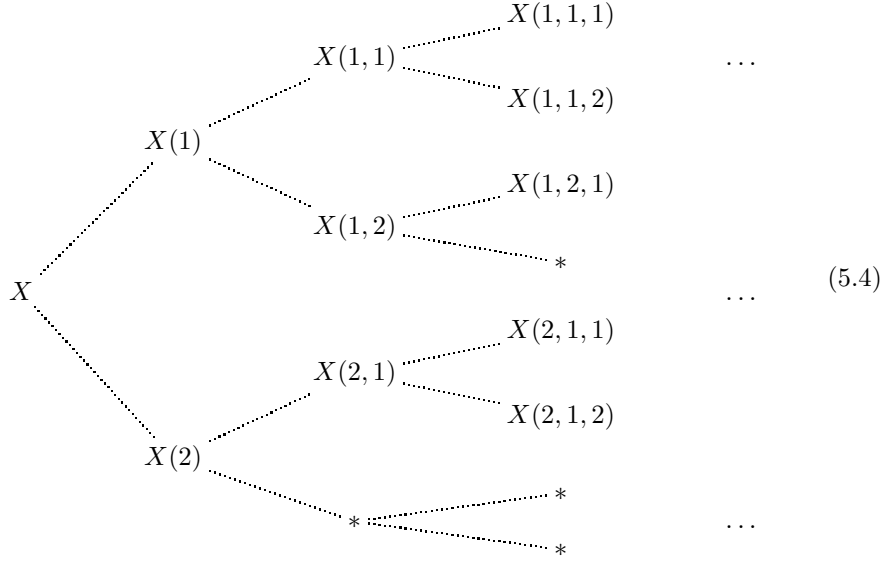
\begin{equation}
\label{equ:tree}
  \vcenter{
    \xymatrix @C+1pc @R-2pc {
      & & & X(1,1,1) & \\
      & & X(1,1) \ar@{.}[ur] \ar@{.}[dr] & & \ldots \\
      & & & X(1,1,2) & \\
      & X(1) \ar@{.}[uur] \ar@{.}[ddr] & & & \\
      & & & X(1,2,1) & \\
      & & X(1,2) \ar@{.}[ur] \ar@{.}[dr] & & \\
      & & & {*} & \\
      X \ar@{.}[uuuur] \ar@{.}[ddddr] & & & & \dots \\
      & & & X(2,1,1) & \\
      & & X(2,1) \ar@{.}[ur] \ar@{.}[dr] & & \\
      & & & X(2,1,2) & \\
      & X(2) \ar@{.}[uur] \ar@{.}[ddr] & & & \\
      & & & {*} & \\
      & & {*} \ar@{.}[ur] \ar@{.}[dr] & & \ldots \\
      & & & {*} & \\
                            }
          }
\end{equation}
The notation is straightforward; for instance, by $X(1,2)$ is denoted
the DG module obtained by first performing the construction of Case (1),
then the construction of Case (2).  The rule for omitting nodes of the
tree is that no $X(\cdots)$  must contain two neighbouring digits $2$.

\begin{theorem}
\label{thm:Schmidt1}
Suppose that $\dim_k \Homology\!R \geq 3$.  Then the AR quiver
$\Gamma$ of $\Dcatc(R)$ has infinitely many components.
\end{theorem}

\begin{proof}
It is a standing assumption in this section that $R$ is Gorenstein, so
each component $C$ of $\Gamma$ is isomorphic to $\BZ A_{\infty}$ as a
stable translation quiver by Theorem
\ref{thm:quiver_local_structure}(1).

Since $\dim_k \Homology\!R \geq 3$, there exists an $e \not\in \{ 0,d
\}$ such that $R$ has non-zero cohomology in degree $e$, so the above
constructions make sense.  Start with $X = R$ and consider the tree
\eqref{equ:tree}.  It follows from Lemma \ref{lem:C}, (1) and (2),
that the function $\varphi$ is constant with value $r$ on the $r$'th
column of the tree.  On the other hand, by Theorem
\ref{thm:quiver_local_structure}(3), the value of $\varphi$ on the
$n$'th horizontal row of a component $C \cong \BZ A_{\infty}$ of
$\Gamma$ is $n\varphi_1$.  Hence, if the vertices corresponding to two
modules in the $r$'th column of the tree \eqref{equ:tree} both belong
to $C$, then they sit in the same horizontal row of vertices in $C$.

Equation \eqref{equ:j} implies that $\amp \tau Y = \amp Y$ for each $Y$
in $\Dcatc(R)$.  However, on $C$, the action of $\tau$ is to move a
vertex one step to the left.  It follows that the amplitude is
constant on each horizontal row of $C$.

Combining these arguments, if the vertices corresponding to two
modules in the $r$'th column of the tree \eqref{equ:tree}
both belong to $C$, then the modules have the same amplitude.

On the other hand, in the construction above, Case (1) makes the
amplitude grow by $d-1$ and Case (2) makes the amplitude grow by
$e-1$.  Let $a_1, \ldots, a_r$ be a sequence of the digits $1$ and $2$
which does not contain two neighbouring digits $2$.  Suppose that the
sequence contains $s$ digits $1$ and $r-s$ digits $2$.  Then since
$\amp X = \amp R = d$ it holds that $\amp X(a_1, \ldots, a_r) = d +
s(d-1) + (r-s)(e-1)$, and since $e < d$ it is clear that this value
changes when $s$ changes.  So by choosing $r$ sufficiently large, a
column of the tree \eqref{equ:tree} can be achieved with an
arbitrarily large number of DG modules with pairwise different
amplitudes.

By the first part of the proof, this results in an arbitrarily large
number of different components of $\Gamma$, so $\Gamma$ has infinitely
many components.
\end{proof}

\begin{theorem}
\label{thm:Schmidt2}
Suppose that there is an $e$ with $\dim_k \Homology^e\!R \geq 2$.
Then the AR quiver $\Gamma$ of $\Dcatc(R)$ has families of distinct
components which are indexed by projective manifolds over $k$, and
these manifolds can be of arbitrarily high dimension.
\end{theorem}

\begin{proof}
Again, it is a standing assumption in this section that $R$ is
Gorenstein, so each component $C$ of $\Gamma$ is isomorphic to $\BZ
A_{\infty}$ as a stable translation quiver by Theorem
\ref{thm:quiver_local_structure}(1).

Set $X = R$.  With an obvious notation, consider
$X(2_{\alpha},1,2_{\beta})$.  Then an isomorphism
$X(2_{\alpha},1,2_{\beta}) \cong
X(2_{\alpha^{\prime}},1,2_{\beta^{\prime}})$ implies $X(2_{\alpha},1)
\cong X(2_{\alpha^{\prime}},1)$ by Lemma \ref{lem:C}(2), and then
$\beta = \lambda\beta^{\prime}$ for a $\lambda$ in $k$ by Lemma
\ref{lem:C}(3).  And $X(2_{\alpha},1) \cong X(2_{\alpha^{\prime}},1)$
implies $X(2_{\alpha}) \cong X(2_{\alpha^{\prime}})$ by Lemma
\ref{lem:C}(1), and then $\alpha = \kappa\alpha^{\prime}$ for a
$\kappa$ in $k$ by Lemma \ref{lem:C}(3).

The $X(2_{\alpha},1,2_{\beta})$ hence give a family of pairwise
non-isomorphic objects of $\Dcatc(R)$ parametrized by the Cartesian
product $\{\mbox{rays of $\alpha$'s}\} \times \{\mbox{rays of
  $\beta$'s}\}$.

Now, $\sup X = d$ so the class $\alpha$ is in $\Homology^{d-d+e}(X)$,
cf.\ the construction in Case (2).  However,
\[
  \Homology^{d-d+e}(X) = \Homology^e(X) = \Homology^e(R).
\]
Hence $\{\mbox{rays of $\alpha$'s}\} = \BP(\Homology^e\!R)$ where
$\BP$ denotes the projective space of rays in a vector space.
Moreover, $\sup X(2_{\alpha},1) = d+(e-1)+(d-1) = 2d+e-2$ by Lemma
\ref{lem:C}, (1) and (2), so the class $\beta$ is in
$\Homology^{(2d+e-2)-d+e}(X(2_{\alpha},1))$.  However,
\begin{align*}
  \Homology^{(2d+e-2)-d+e}(X(2_{\alpha},1))
  & = \Homology^{d+2e-2}(X(2_{\alpha},1)) \\
  & = \Homology^{(d+e-1)+e-1}(X(2_{\alpha},1)) \\
  & \cong \Homology^e(R),
\end{align*}
where $\cong$ is by Lemma \ref{lem:C}(1) because $\sup X(2_{\alpha}) =
d+e-1$.  Hence it is also the case that $\{\mbox{rays of $\beta$'s}\}
= \BP(\Homology^e\!R)$.

This shows that the $X(2_{\alpha},1,2_{\beta})$ give a family of
pairwise non-i\-so\-mor\-phic objects of $\Dcatc(R)$ indexed by
$\BP(\Homology^e\!R) \times \BP(\Homology^e\!R)$.  Note that the
projective space $\BP(\Homology^{e}\!R)$ is non-trivial since $\dim_k
\Homology^e\!R \geq 2$.

By Lemma \ref{lem:C}, (1) and (2), all the $X(2_{\alpha},1,2_{\beta})$
have the same value of $\varphi$ (it is $4$), so if the vertices of
two non-isomorphic ones belonged to the same component $C$ of
$\Gamma$, then they would be different vertices in the same horizontal
row of $C \cong \BZ A_{\infty}$ because the value of $\varphi$ on the
$n$'th row of $C$ is $n\varphi_1$ by Theorem
\ref{thm:quiver_local_structure}(3).  However, it follows from
Equation \eqref{equ:j} that $\inf(\tau Y) = \inf(Y) - d + 1$, so
different vertices in the $n$'th row of $C$ correpond to DG modules
with different $\inf$, but the $X(2_{\alpha},1,2_{\beta})$ all have
the same $\inf$ by Lemma \ref{lem:C}, (1) and (2) (it is $0$).  Hence
the vertices of two non-isomorphic $X(2_{\alpha},1,2_{\beta})$'s must
belong to different components of $\Gamma$, so a family has been found
of distinct components of $\Gamma$ parametrized by the projective
manifold $\BP(\Homology^e\!R) \times \BP(\Homology^e\!R)$ over $k$.

An analogous argument with objects of the form
$X(2_{\alpha},1,2_{\beta},1,\ldots,1,2_{\gamma})$ produces families of
distinct components of the AR quiver indexed by projective
ma\-ni\-folds of arbitrarily high dimension, as claimed.
\end{proof}

\section{Poincar\'{e} duality spaces}
\label{sec:topology}

This section makes explicit the highlights of the previous sections
for DG algebras of the form $\Chain^*(X;k)$.  The results first
appeared in \cite{artop}, \cite{arquiv}, and \cite{Schmidt}.

\begin{setup}
In this section, the field $k$ will have characteristic $0$.  By $X$
will be denoted a simply connected topological space with $\dim_k
\Homology^*(X;k) < \infty$.  Write
\[
  n = \sup \{\, i \,|\, \Homology^i(X;k) \neq 0 \,\}.
\]

When the singular cochain complex $\Chain^*(X;k)$ and singular
cohomology $\Homology^*(X;k)$ appear below, it is always with
coefficients in $k$, so I will use the shorthands $\Chain^*(X)$ and
$\Homology^*(X)$.
\end{setup}

\begin{remark}
\label{rmk:A}
The singular chain complex $\Chain^*(X)$ is a DG algebra under cup
product, and by \cite[exa.\ 6, p.\ 146]{FHTbook}, it is
quasi-isomorphic to a commutative DG algebra $A$ satisfying the
conditions of Setup \ref{set:blanket}.
\end{remark}

\begin{remark}
For $X$ to be simply connected means that it is path connected and
that each closed path in $X$ can be shrinked continuously to a point.
Equivalently, $X$ is path connected and its fundamental group
$\pi_1(X)$ is trivial.

The space $X$ is said to have Poincar\'{e} duality over $k$ if there
is an isomorphism
\[
  \dual\!\Homology^*(X) \cong \Sigma^n \Homology^*(X)
\]
of graded left-$\Homology^*(X)$-modules.  It is a classical theorem
that any compact $n$-dimensional manifold has Poincar\'{e} duality;
indeed, this is one of the oldest results of algebraic topology.

A consequence of Poincar\'{e} duality over $k$ is that there are
isomorphisms of vector spaces
\[
  \dual\!\Homology^i(X) \cong \Homology^{n-i}(X)
\]
for each $i$, and hence that the singular cohomology $\Homology^*(X)$
with coefficients in $k$ is concentrated between dimensions $0$ and
$n$ and has the same vector space dimension in degrees $i$ and $n-i$.
Geometrically, this is in a sense the statement that the number of
holes with $i$-dimensional boundary enclosed by $X$ is equal to the
number of holes with $(n-i)$-dimensional boundary enclosed by $X$.

Algebraically, spaces with Poincar\'{e} duality emulate Gorenstein
algebras; see \cite{FHTpaper}.
\end{remark}

For the definition of $n$-Calabi-Yau categories, see Definitions
\ref{def:Serre} and \ref{def:CY}.

\begin{theorem}
\label{thm:Chain_CY}
The following conditions are e\-qui\-va\-lent.
\begin{enumerate}
  \item  $\Dcatc(\Chain^*(X))$ is an $n$-Calabi-Yau category.
  \item  $\Dcatc(\Chain^*(X)^{\opp})$ is an $n$-Calabi-Yau category.
  \item  $X$ has Poincar\'{e} duality over $k$.
\end{enumerate}
\end{theorem}

\begin{proof}
This will involve showing that the conditions of the theorem are
also e\-qui\-va\-lent to the following two conditions.
\begin{enumerate}
  \setcounter{enumi}{3}
  \item  $\Dcatc(\Chain^*(X))$ has AR triangles.
  \item  $\Dcatc(\Chain^*(X)^{\opp})$ has AR triangles.
\end{enumerate}

For the proof, $\Chain^*(X)$ can be replaced with the commutative DG
algebra $A$ by Remark \ref{rmk:A}.  So it is clear that
(1)$\Leftrightarrow$(2) and that (4)$\Leftrightarrow$(5).

Condition (3), that $X$ has Poincar\'{e} duality, means
${}_{\Homology\!A}(\dual\!\Homology\!A) \cong {}_{\Homology\!A}(\Sigma^n 
\Homology\!A)$; since $A$ is commutative, Theorem
\ref{thm:Gorenstein}(2) implies that this is equivalent to $A$ being
Gorenstein.  Condition (4) is also equivalent to $A$ being Gorenstein
by Theorem \ref{thm:R}.  It follows that (3)$\Leftrightarrow$(4).

(1)$\Rightarrow$(4) holds since a Calabi-Yau category has a Serre
functor and hence AR triangles, see Definition \ref{def:Serre},
Theorem \ref{thm:Serre}, and Definition \ref{def:CY}.

(3)$\Rightarrow$(1).  The DG algebra $A$ is commutative, so Theorem
\ref{thm:Gorenstein}(3) implies that condition (3) is equivalent to
\[
  \dual\!A \cong \Sigma^n A
\]
in the derived category of DG bi-$A$-modules.  Inserting this into
Equation \eqref{equ:Serre2} shows that the Serre functor of $\Dcatc(A)$
is $\Sigma^n$ so (1) holds, cf.\ Definition \ref{def:CY}.
\end{proof}

\begin{theorem}
\label{thm:Chain_Gamma}
Suppose that $X$ has Poincar\'{e} duality over $k$ and that it
satisfies $\dim_k \Homology^*(X) \geq 2$.  Then each component of the
AR quiver $\Gamma$ of $\Dcatc(\Chain^*(X))$ is isomorphic to $\BZ
A_{\infty}$.

If $\dim_k \Homology^*(X) = 2$, then $\Gamma$ has $n-1$ components. 

If $\dim_k \Homology^*(X) \geq 3$, then $\Gamma$ has infinitely many
components.

If $\dim_k \Homology^e(X) \geq 2$ for some $e$, then $\Gamma$ has
families of distinct components which are indexed by projective
manifolds over $k$, and these manifolds can be of arbitrarily high
dimension.
\end{theorem}

\begin{proof}
Since $\Chain^*(X)$ is quasi-isomorphic to $A$, the theory of the
previous sections applies to $\Chain^*(X)$.  As in the proof of
Theorem \ref{thm:Chain_CY}, since $X$ has Poincar\'{e} duality,
$\Chain^*(X)$ is Gorenstein.  The present theorem hence follows from
Theorems \ref{thm:quiver_local_structure}, \ref{thm:spheres},
\ref{thm:Schmidt1}, and \ref{thm:Schmidt2}.
\end{proof}

Theorem \ref{thm:Chain_CY} and its proof imply that if $X$ has
Poincar\'{e} duality over $k$, then the AR quiver of
$\Dcatc(\Chain^*(X))$ is a stable translation quiver.

\begin{theorem}
The AR quiver of $\Dcatc(\Chain^*(X))$ is a weak homotopy invariant of
$X$.

If $X$ is restricted to spaces with Poincar\'{e} duality over $k$,
then the AR quiver of $\Dcatc(\Chain^*(X))$, viewed as a stable
translation quiver, is a weak homotopy invariant of $X$.
\end{theorem}

\begin{proof}
If $X$ and $X^{\prime}$ have the same weak homotopy type, then by
\cite[thm.\ 4.15]{FHTbook} there exists a series of quasi-isomorphisms
of DG algebras linking $\Chain^*(X)$ and $\Chain^*(X^{\prime})$.
Hence $\Dcatc(\Chain^*(X))$ and $\Dcatc(\Chain^*(X^{\prime}))$ are
equivalent triangulated categories, and this implies both parts of the
theorem.
\end{proof}

\section{Open problems}
\label{sec:open}

Let me close the paper by proposing the following open problems.  The
first one is due to Karsten Schmidt, see \cite[sec.\ 6]{Schmidt}.

\begin{problem}
Develop a theory of representation type of simply connected co\-chain DG
algebras.

What is known so far is the following.
\begin{enumerate}
  \item  By Theorem \ref{thm:spheres}, if $\dim_k \Homology\!R = 2$,
    then the AR quiver $\Gamma$ of $\Dcatc(R)$ has a finite number of
    components. 

    Suppose that $R$ is Gorenstein.

  \item  By Theorem \ref{thm:Schmidt1}, if $\dim_k \Homology\!R \geq
    3$, then $\Gamma$ has infinitely many components.
    
  \item  By Theorem \ref{thm:Schmidt2}, if $\dim_k \Homology^e\!R \geq
    2$ for some $e$, then $\Gamma$ has families of distinct components
    which are indexed by projective manifolds, and these manifolds can
    be of arbitrarily high dimension.
\end{enumerate}
It is tempting to interpret the DG algebras of (1) as having finite
representation type, and the ones of (3) as having wild representation
type.

If $\dim_k \Homology\!R \geq 3$ but $\dim_k \Homology^i\!R \leq 1$ for
each $i$, then it is not clear whether the infinitely many components
of $\Gamma$ form discrete or continuous families, or indeed, what
these words precisely mean in the context.

Note that some previous work does exist on the representation type of
derived categories, see \cite{GeissKrause}, but it does not apply to
the categories considered in this paper.
\end{problem}

\begin{problem}
What is the structure of the AR quiver of $\Dcatc(R)$ if $R$ is not
Gorenstein?

Do components of a different shape than $\BZ A_{\infty}$ become
possible?
\end{problem}

\begin{problem}
Generalize the theory to cochain DG algebras which are not simply
connected.

Presently, not even the structure of $\Dcatc(\Chain^*(S^1;\BQ))$ is
known because $S^1$ and hence $\Chain^*(S^1;\BQ)$ is not simply
connected.

A generalization to the non-simply connected case may impact on
non-com\-mu\-ta\-ti\-ve geometry for which more general cochain DG
algebras are being considered as vehicles.
\end{problem}

\begin{problem}
Is there a link between the categories $\Dcatc(R)$ which have AR
qui\-vers consisting of $\BZ A_{\infty}$-components, and the appearence
of $\BZ A_{\infty}$-components in representation theory?

See for instance \cite[thm.\ 4.17.4]{BensonI}.
\end{problem}

\begin{problem}
If a simply connected topological space $X$ has
$\dim_{\BQ}\Homology^*(X;\BQ) = 2$, then it has the same rational
homotopy type as a sphere of dimension $\geq 2$.  Theorem
\ref{thm:Chain_Gamma} implies that these are the only simply connected
spaces with Poincar\'{e} duality for which the AR quiver of
$\Dcatc(\Chain^*(X;\BQ))$ has only finitely many components.

Is this linked to any topological property which is special to these
spaces?
\end{problem}

\begin{problem}
Let $X$ and $T$ be topological spaces.  Suppose that $X$ is simply
connected with $\dim_k \Homology^*(X;k) < \infty$, that $T$ has
$\dim_k \Homology^i(T;k) < \infty$ for each $i$, and let
\[
  F \rightarrow T \rightarrow X
\]
be a fibration.  The induced morphism $\Chain^*(X;k) \rightarrow
\Chain^*(T;k)$ turns $\Chain^*(T;k)$ into a DG
left-$\Chain^*(X;k)$-module.  By \cite[thm.\ 7.5]{FHTbook} there is a
quasi-isomorphism $k \LTensor_{\Chain^*(X;k)} \Chain^*(T;k) \simeq
\Chain^*(F;k)$, and this implies that if $\dim_k \Homology^*(F;k) <
\infty$ then $\Chain^*(T;k)$ is an object of $\Dcatc(\Chain^*(X;k))$.

Hence $\Chain^*(T;k)$ corresponds to a collection of vertices with
multiplicities of the AR quiver $\Gamma$ of $\Dcatc(\Chain^*(X;k))$.
If $X$ has Poincar\'{e} duality over $k$, then the theory of this
paper gives information about the structure of $\Gamma$, both locally
and globally.

Does this have applications to the topological theory of fibrations?

Do the structural results on $\Gamma$ correspond to structural
results on topological fibrations?
\end{problem}

\begin{problem}
By considering the fibration $F \rightarrow T \rightarrow X$, looking
at $\Chain^*(T;k)$ as a DG left-$\Chain^*(X;k)$-module, and using the
theory of this paper, one is in effect doing ``AR theory
with topological spaces''.

Is there a way to do so directly with the spaces themselves?
\end{problem}

\begin{problem}
\label{prob:CY}
If $X$ is a topological space with $\dim_k \Homology^*(X;k) < \infty$
and Poincar\'{e} duality over the field $k$ of characteristic $0$,
then $\Dcatc(\Chain^*(X;k))$ is an $n$-Calabi-Yau category for some
$n$ by Theorem \ref{thm:Chain_CY}.  More generally, if $R$ is the DG
algebra from setup \ref{set:blanket} and $R$ is commutative and
Gorenstein, then $\Dcatc(R)$ is a $d$-Calabi-Yau category.

These categories appear to behave quite differently from higher cluster
categories which are standard examples of Calabi-Yau categories.  For
instance, an $m$-cluster category contains an $m$-cluster tilting
object in terms of which every other object can be built in a single
step; this seems to be far from true for $\Dcatc(\Chain^*(X;k))$ and
$\Dcatc(R)$.

Which role do $\Dcatc(\Chain^*(X;k))$ and $\Dcatc(R)$ play in the
taxonomy of Calabi-Yau categories?

In the context of Calabi-Yau categories, there is a ``Morita'' theorem
for higher cluster categories, see \cite[thm.\ 4.2]{KellerReiten}.  Is
there also a Morita theorem for the categories $\Dcatc(R)$?
\end{problem}

\appendix

\section{Differential Graded homological algebra}
\label{app:DG}

This appendix is an introduction to Differential Graded (DG)
homological algebra, written for a reader who is already familiar with
the formalism of derived categories of rings.  Some useful references
are \cite{AFH}, \cite[appendix]{FHTpaper}, \cite[chps.\ 3, 6, 18, 19,
20]{FHTbook}, \cite{KellerDG}, and \cite{KellerICM}.

Let $k$ be a commutative ring.

\begin{definition}
[DG algebras and modules]
\label{def:DG}
A Differential Graded (DG) algebra $R$ over $k$ is a complex of
$k$-modules equipped with a product which
\begin{itemize}
  \item  turns $R$ into a $\BZ$-graded $k$-algebra, and
  \item  satisfies the Leibniz rule $\partial^R(rs) = \partial^R(r)s +
    (-1)^i r\partial^R(s)$ when $r$ is in $R^i$.
\end{itemize}

A DG left-$R$-module $M$ is a complex of $k$-modules equipped with an
$R$-scalar multiplication which
\begin{itemize}
  \item  turns it into a graded module over the underlying graded
    algebra of $R$, and
  \item  satisfies the Leibniz rule $\partial^M(rm) = \partial^R(r)m +
    (-1)^i r\partial^M(m)$ when $r$ is in $R^i$.
\end{itemize}

DG right-$R$-modules and DG bi-modules are defined analogously.  Note
that $R$ itself is an important DG bi-$R$-module.  Sometimes the
notations ${}_{R}M$ and $N_R$ are used to emphasize that $M$ is a DG
left-$R$-module, $N$ a DG right-$R$-module.

The opposite DG algebra of $R$ is denoted by $R^{\opp}$.  Its product
$\cdot$ is given by $r \cdot s = (-1)^{ij}sr$ in terms of the product
of $R$, when $r$ and $s$ are elements of $R^i$ and $R^j$.  DG
right-$R$-modules can be viewed as DG left-$R^{\opp}$-modules.
\end{definition}

\begin{remark}
[DG homological algebra]
It is possible to do homological algebra with DG modules.  A test case
is when the DG algebra $R$ is concentrated in degree zero, that is,
when $R^i = 0$ for $i \neq 0$.  Then the zeroth component, $R^0$, is
an ordinary $k$-algebra, DG left-$R$-modules are just complexes of
left-$R^0$-modules, and DG homological algebra over $R$ specializes to
ordinary homological algebra over $R^0$.
\end{remark}

\begin{definition}
[inf, sup, and amp]
\label{def:sup_and_inf}
The infimum and supremum of a DG module are
\[
  \inf M = \inf \{\, i \,|\, \Homology^i\!M \neq 0 \,\}, \;\;\;
  \sup M = \sup \{\, i \,|\, \Homology^i\!M \neq 0 \,\},
\]
and the amplitude is
\[
  \amp M = \sup M - \inf M.
\]
Note that $\inf 0 = \infty$, $\sup 0 = -\infty$, and $\amp 0 =
-\infty$. 
\end{definition}

\begin{definition}
[Morphisms, suspensions, and mapping cones]
\label{def:morphisms}
The notation $(-)^{\natural}$ is used for the operation of forgetting
the differential.  It sends DG algebras and DG modules to graded
algebras and graded modules.

A morphism $\rho : R \rightarrow S$ of DG algebras is a homomorphism
$R^{\natural} \rightarrow S^{\natural}$ of the underlying graded
algebras which respects the differentials, $\rho\partial^R =
\partial^S\rho$.

A morphism $\mu : M \rightarrow N$ of DG $R$-modules is a homomorphism
$M^{\natural} \rightarrow N^{\natural}$ of the underlying graded
$R^{\natural}$-modules which respects the differentials,
$\mu\partial^M = \partial^N \mu$.

The morphism $\mu$ is called null homotopic if there exists a
homomorphism $\theta : M^{\natural} \rightarrow N^{\natural}$ of
degree $-1$ of graded $R^{\natural}$-modules such that $\mu =
\partial^N\theta + \theta\partial^M$.  Morphisms $\mu$ and
$\mu^{\prime}$ are called homotopic if $\mu - \mu^{\prime}$ is null
homotopic.

Suspension of complexes is denoted by $\Sigma$.  Suspensions and
mapping cone constructions of DG left-$R$-modules inherit DG
left-$R$-module structures.  Some sign issues are involved here as
well as in other parts of the theory; I will not go into details but
refer the reader to the references given.
\end{definition}

\begin{definition}
[Cohomology]
\label{def:homology}
The product on $R$ and the scalar multiplication of $R$ on $M$ induces
a product on the cohomology $\Homology\!R$ and a scalar multiplication
of $\Homology\!R$ on $\Homology\!M$, whereby $\Homology\!R$ becomes a
graded $k$-algebra and $\Homology\!M$ becomes a graded
$\Homology\!R$-module.

A morphism $\mu$ of DG modules is called a quasi-isomorphism if the
induced homomorphism $\Homology\!\mu$ of graded $\Homology\!R$-modules
is an isomorphism.
\end{definition}

\begin{definition}
[Homotopy and derived categories]
\label{def:D}
The homotopy category $\sK(R)$ has as objects the DG left-$R$-modules,
and as morphisms the homotopy classes of morphisms of DG modules.

The derived category $\Dcat(R)$ is obtained from $\sK(R)$ by formally
inverting the quasi-isomorphisms.  Both $\sK(R)$ and $\Dcat(R)$ are
triangulated categories with distinguished triangles induced by the
mapping cone construction.

The categories $\sK(R)$ and $\Dcat(R)$ have set indexed coproducts
which are given by ordinary direct sums.

The categories $\sK(R^{\opp})$ and $\Dcat(R^{\opp})$ can be viewed as
being the homotopy and derived categories of DG right-$R$-modules.

A quasi-isomorphism $R \rightarrow S$ of DG algebras induces an
equivalence of triangulated categories $\Dcat(S) \rightarrow
\Dcat(R)$ given by change of scalars. 

Denote by $\Dcatf(R)$ the full subcategory of $\Dcat(R)$ consisting of
DG modules $M$ with $\Homology\!M$ finitely presented over $k$.

Denote by $\Dcatc(R)$ the full subcategory of $\Dcat(R)$ consisting of
DG modules which are finitely built in $\Dcat(R)$ from $R$ using
distinguished triangles, (de)suspensions, coproducts, and direct
summands; these are the so-called compact objects of $\Dcat(R)$.
\end{definition}

\begin{definition}
[Hom and Tensor]
\label{def:Hom_and_Tensor}
If $M$ and $N$ are DG left-$R$-modules, then there is a graded
$k$-module $\Hom_{R^{\natural}}(M^{\natural},N^{\natural})$ of graded
$R^{\natural}$-homomorphisms $M^{\natural} \rightarrow N^{\natural}$
of different degrees.  This can be turned into a complex $\Hom_R(M,N)$
with the differential induced by the differentials of $M$ and $N$.
Note that $\Hom_R(M,N)^{\natural} =
\Hom_{R^{\natural}}(M^{\natural},N^{\natural})$.

If $A$ is a DG right-$R$-module and $B$ is a DG left-$R$-module, then
the tensor product $A^{\natural} \otimes_{R^{\natural}} B^{\natural}$
is a graded $k$-module.  It can be turned into a complex $A \otimes_R
B$ with the differential induced by the differentials of $A$ and $B$.
Note that $(A \otimes_R B)^{\natural} = A^{\natural}
\otimes_{R^{\natural}} B^{\natural}$.

These constructions induce functors between homotopy categories, and
there are induced derived functors
\[
  \RHom_R(-,-) : \Dcat(R) \times \Dcat(R) \rightarrow \Dcat(k)
\]
and
\[
  - \LTensor_R - : \Dcat(R^{\opp}) \times \Dcat(R) \rightarrow \Dcat(k).
\]
These are often computed using resolutions.  For instance, let $M$ be
a DG left-$R$-module and let $P \rightarrow M$ be a K-projective
resolution of $M$.  This is a quasi-isomorphism of DG modules for
which $P$ is K-projective, that is, $\Hom_R(P,-)$ preserves
quasi-isomorphisms.  Then $\Hom_R(P,-)$ is a well defined functor
$\Dcat(R) \rightarrow \Dcat(k)$, and there is an equivalence of
functors $\RHom_R(M,-) \simeq \Hom_R(P,-)$.

The functor $\RHom_R$ has the useful property
\[
  \Homology^0 \RHom_R(M,N) \cong \Hom_{\Dcat(R)}(M,N);
\]
more generally, the notation
\[
  \Homology(\RHom_R(M,N)) = \Ext_R(M,N)
\]
is used so $\Homology^i \RHom_R(M,N) = \Ext_R^i(M,N)$.

The functors $\RHom$ and $\LTensor$ are compatible with DG bi-modules.
For instance, if $A$ is a DG bi-$R$-module then $A \LTensor_R B$
inherits a left-$R$-structure from $A$, so there is a functor
\[
  A \LTensor_R - : \Dcat(R) \rightarrow \Dcat(R).
\]
\end{definition}

\begin{setup}
Now consider the special case of this paper: $k$ is a field and $R$ is
a DG algebra over $k$ which has the form
\[
  \cdots \rightarrow 0 \rightarrow k \rightarrow 0 \rightarrow R^2
  \rightarrow R^3 \rightarrow \cdots,
\]
that is, $R^{<0} = 0$, $R^0 = k$, and $R^1 = 0$.  It will also be
assumed that $\dim_k R < \infty$, and $d$ will be defined by
\[
  d = \sup R.
\]
\end{setup}

\begin{definition}
[Duality]
\label{def:dual}
By $\dual(-)$ will be denoted the functor $\Hom_k(-,k)$.  When applied
to graded objects, it is understood to be applied degreewise.  It
sends DG left-$R$-modules to DG right-$R$-modules and vice versa.  It
is well defined at the level of homotopy and derived categories.
\end{definition}

\begin{remark}
\label{rmk:Df}
Over a DG algebra of the present special form, $k \cong R/R^{\geq 1}$
is a DG bi-$R$-module.  Moreover, $\Dcatf(R)$ is the full subcategory
of $\Dcat(R)$ consisting of objects $M$ with $\dim_k \Homology\!M <
\infty$, and $\Dcatf(R)$ consists precisely of the objects finitely
built from ${}_{R}k$.  This can be shown using the first two parts of
the following result on truncations, the proof of which uses only
linear algebra over the field $k$; see \cite[lem.\ 3.4]{artop} and
\cite[ex.\ 6, p.\ 146]{FHTbook}.
\end{remark}

\begin{lemma}
[Truncations]
\label{lem:truncations}
\begin{enumerate}

  \item  If $M$ is a DG left-$R$-module with $\inf M$ finite, then
    there exists a quasi-isomorphism of DG left-$R$-modules $U
    \rightarrow M$ with $U^i = 0$ for $i < \inf M$.

  \item  If $N$ is a DG left-$R$-module with $\sup N$ finite, then
    there exists a quasi-isomorphism of DG left-$R$-modules $N
    \rightarrow V$ with $V^j = 0$ for $j > \sup N$.

  \item  The DG algebra $R$ is quasi-isomorphic to a quotient DG
    algebra $S$ with $S^{>d} = 0$.

\end{enumerate}
\end{lemma}

\begin{definition}
[Semi-free DG modules]
A DG left-$R$-module $F$ is called semi-free if it permits a semi-free
filtration, that is, a filtration by DG left-$R$-modules
\[
  0 = F \langle -1 \rangle
  \subseteq F \langle 0 \rangle
  \subseteq F \langle 1 \rangle
  \subseteq \cdots \subseteq F
\]
where $F = \bigcup_i F \langle i \rangle$ and where each $F \langle i
\rangle / F \langle i-1 \rangle$ is a direct sum of (de)suspensions of
${}_{R}R$.

If $\partial^F(F) \subseteq R^{\geq 1} \cdot F$, then $F$ is called
minimal.  If $M$ is in $\Dcat(R)$ then a (minimal) semi-free
resolution of $M$ is a quasi-isomorphism $F \rightarrow M$ where $F$
is (minimal) semi-free.
\end{definition}

The following lemma collects useful facts; for references see
\cite{AFH}, \cite{BN}, \cite[appendix]{FHTpaper}, \cite[sec.\ 
6]{FHTbook}, \cite[sec.\ 3]{artop}, \cite[sec.\ 3]{KellerDG}, and
\cite{Spaltenstein}.

\begin{lemma}
[Semi-free resolutions]
\label{lem:semi-free}
\begin{enumerate}

  \item  Each $M$ in $\Dcat(R)$ has a semi-free re\-so\-lu\-ti\-on.

  \item  A semi-free DG module is {\rm K}-projective, so if $F$ is a
    semi-free resolution of $M$ then $\RHom_R(M,-) \simeq
    \Hom_R(F,-)$ and $- \LTensor_R M \simeq - \otimes_R F$.
    
  \item  Each $M$ in $\Dcatf(R)$ has a minimal semi-free resolution
    $F$, and for each such resolution there are finite numbers
    $\beta_i$ such that
\[
  F^{\natural} \cong \bigoplus_{i \leq -\inf M}
                     \Sigma^i(R^{\natural})^{(\beta_i)},
\]
    where $(R^{\natural})^{(\beta_i)}$ is a direct sum of $\beta_i$
    copies of $R^{\natural}$.

  \item  Let $M$ in $\Dcatf(R)$ have minimal semi-free resolution $F$.
    Then $M$ is in $\Dcatc(R)$ if and only the numbers $\beta_i$ from
    part (3) satisfy $\beta_i = 0$ for $i \ll 0$.
    
  \item  If $F$ is minimal semi-free then $\Hom_R(F,k)$ has zero
    differential, so
\[
  \Homology(\Hom_R(F,k))
  \cong \Hom_R(F,k)^{\natural} 
  = \Hom_{R^{\natural}}(F^{\natural},k^{\natural})
\]
as graded $k$-vector spaces.

\end{enumerate}
\end{lemma}

\section{Auslander-Reiten theory for triangulated categories}
\label{app:AR}

This appendix is a brief introduction to the version of
Auslander-Reiten (AR) theory used in the rest of the paper.  Some
useful references are \cite{ARS}, \cite{BensonI}, \cite{Happel},
\cite{KrauseARZ}, \cite{KrauseAR}, and \cite{ReitenVandenBergh}, with
\cite{Happel} being the source of the theory.

Let $\sT$ be a triangulated category.  The following definition is
taken from \cite[def.\ 2.1]{KrauseAR}; it generalizes the earlier
definition from \cite[3.1]{Happel}.

\begin{definition}
[AR triangles]
An AR triangle in $\sT$ is a distinguished triangle
\begin{equation}
\label{equ:AR}
  M
  \stackrel{\mu}{\rightarrow} N
  \stackrel{\nu}{\rightarrow} P
  \stackrel{\pi}{\rightarrow}
\end{equation}
for which
\begin{itemize}

  \item  Each morphism $M \rightarrow N^{\prime}$ which is not a split
    monomorphism factors through $\mu$.

  \item  Each morphism $N^{\prime} \rightarrow P$ which is not a split
    epimorphism factors through $\nu$.

  \item  $\pi \neq 0$.

\end{itemize}
\end{definition}

In an AR triangle, the end terms determine each other up to
isomorphism by \cite[prop.\ 3.5(i)]{Happel}, so the following
definition makes sense.

\begin{definition}
[AR translation]
Let $P$ be an object of $\sT$ and suppose that there is an AR triangle
$M \rightarrow N \rightarrow P \rightarrow$.  Then $M$ is denoted by
$\tau P$, and the operation $\tau$ which is defined up to isomorphism
is called the AR translation of $\sT$.  
\end{definition}

In an AR triangle, the end terms have local endomorphism rings by
\cite[lem.\ 2.3]{KrauseAR}; this explains the following terminology.

\begin{definition}
The triangulated category $\sT$ is said to have right AR triangles if,
for each object $P$ with local endomorphism ring, there is an 
AR triangle \eqref{equ:AR}.

The category $\sT$ is said to have left AR triangles if, for each
object $M$ with local endomorphism ring, there is an AR triangle
\eqref{equ:AR}.

The category $\sT$ is said to have AR triangles if it has right and
left AR triangles.
\end{definition}

\begin{definition}
[The AR quiver]
A morphism in $\sT$ is called irreducible if it is not an isomorphism,
but has the property that when it is factored as $\rho\sigma$, then
either $\rho$ is a split epimorphism or $\sigma$ is a split
monomorphism.

The AR quiver $\Gamma(\sT)$ of $\sT$ has one vertex $[M]$ for each
isomorphism class of objects with local endomorphism rings, and one
arrow $[M] \rightarrow [N]$ when there is an irreducible morphism $M
\rightarrow N$.

If $\sT$ has right AR triangles, then the AR translation $\tau$
induces a map from the set of vertices of $\Gamma(\sT)$ to itself.  By
abuse of notation, this map is also referred to as the AR translation
and denoted by $\tau$.
\end{definition}

\begin{setup}
Now consider the special case of this paper: $k$ is a field and
$\sT$ is $k$-linear and has finite dimensional $\Hom$ spaces and
split idempotents; cf.\ Proposition \ref{pro:DcandDfKrullSchmidt}. 

Then $\sT$ is a Krull-Schmidt category by \cite[p.\ 52]{Ringel}; that
is, each indecomposable object has local endomorphism ring and each
object splits into a finite direct sum of indecomposable objects which
are unique up to isomorphism.
\end{setup}

The following lemma holds by \cite[prop.\ 3.5]{Happel}.

\begin{lemma}
\label{lem:Happel}
Let $M \rightarrow N \rightarrow P \rightarrow$ be an AR triangle
and let $N \cong \coprod_i N_i$ where each $N_i$ is indecomposable.
Then the following statements are equivalent for an indecomposable
object $N^{\prime}$.
\begin{enumerate}

  \item  There is an irreducible morphism $M \rightarrow N^{\prime}$.

  \item  There is an irreducible morphism $N^{\prime} \rightarrow P$.

  \item  There is an $i$ such that $N^{\prime} \cong N_i$.

\end{enumerate}
\end{lemma}

Hence if $\sT$ has AR triangles, knowledge of these triangles implies
knowledge of the AR quiver $\Gamma(\sT)$.

\begin{definition}
[Stable translation quivers]
\label{def:stable_translation_quiver}
A stable translation quiver is a quiver equipped with an injective map
$\tau$ from the set of vertices to itself such that the number of
arrows from $\tau(t)$ to $s$ is equal to the number of arrows from $s$
to $t$.
\end{definition}

The following proposition follows easily from Lemma \ref{lem:Happel}.

\begin{proposition}
\label{pro:stable_translation_quiver}
If $\sT$ has AR triangles, then the AR translation $\tau$ turns the AR
quiver $\Gamma(\sT)$ into a stable translation quiver.
\end{proposition}

\begin{definition}
[Serre functors]
\label{def:Serre}
A Serre functor for $\sT$ is an autoequivalence $S$ for which there
are natural isomorphisms
\[
  \dual(\Hom_{\sT}(M,N)) \cong \Hom_{\sT}(N,SM).
\]
\end{definition}

The following was proved in \cite[thm.\ I.2.4]{ReitenVandenBergh}.

\begin{theorem}
\label{thm:Serre}
The category $\sT$ has AR triangles if and only if it has a Serre
functor $S$.  If it does, then $\tau = \Sigma^{-1}S$ on indecomposable
objects.
\end{theorem}

This implies that if $\sT$ has AR triangles, then the AR translation
$\tau$ can be extended to the autoequivalence $\Sigma^{-1}S$.

\begin{definition}
[Calabi-Yau categories]
\label{def:CY}
The category $\sT$ is called $n$-Calabi-Yau if $n$ is the smallest
non-negative integer for which $\Sigma^n$ is a Serre functor.
\end{definition}


\frenchspacing
\begin{thebibliography}{99}

\bibitem{ARS}  M.\ Auslander, I.\ Reiten, and S.\ Smal\o,
  \emph{Representation theory of Artin algebras}.  Cambridge Stud.\
  Adv.\ Math.\ 36, Cambridge University Press, Cambridge, 1997,
  paperback edition. 

\bibitem{AFH}  L.\ L.\ Avramov, H.-B.\ Foxby, and S.\ Halperin,
  Differential Graded homological algebra.  Preprint, version from 21
  June 2006.

\bibitem{BensonI}  D.\ J.\ Benson, \emph{Representations and
    cohomology I}.  Cambridge Stud.\ Adv.\ Math.\ 30, Cambridge
  University Press, Cambridge, 1995.

\bibitem{BN}  M.\ B\"{o}kstedt and A.\ Neeman, Homotopy limits in
  triangulated categories.  \emph{Compositio Math.} \textbf{86} (1993),
  209--234. 

\bibitem{FHTpaper}   Y.\ F\'elix, S.\ Halperin, and J.-C.\ Thomas,
  Gorenstein spaces.  \emph{Adv.\ Math.} \textbf{71} (1988), 92--112. 

\bibitem{FHTbook}  Y.\ F\'elix, S.\ Halperin, and J.-C.\ Thomas,
\emph{Rational Homotopy Theory}.  Grad.\ Texts in Math.\ 205,
Springer, Berlin, 2001.

\bibitem{FJcochain}  A.\ J.\ Frankild and P.\ J\o rgensen, Homological
  properties of cochain Differential Graded algebras.  Preprint
  (2008).  {\tt math.RA/0801.1581.}

\bibitem{GeissKrause}  C.\ Geiss and H.\ Krause, On the
  notion of derived tameness.  \emph{J.\ Algebra Appl.} {\bf 1} (2002),
  133--157.

\bibitem{Happel}  D.\ Happel, On the derived category of a finite
  dimensional algebra.  \emph{Comment.\ Math.\ Helv.} \textbf{62}
  (1987), 339--389.

\bibitem{HKR}  D.\ Happel, B.\ Keller, and I.\ Reiten, Bounded derived
  categories and repetitive algebras.  Preprint (2007).  {\tt
    math.RT/0702302.} 

\bibitem{artop}  P.\ J\o rgensen, Auslander-Reiten theory over
  topological spaces.  \emph{Comment.\ Math.\ Helv.} \textbf{79}
  (2004), 160--182.

\bibitem{arquiv}  P.\ J\o rgensen, The Auslander-Reiten quiver of a
  Poincar\'{e} duality space.  \emph{Al\-ge\-br.\ Re\-pre\-sent.\
    Theory} \textbf{9} (2006), 323--336. 

\bibitem{KellerDG}  B.\ Keller, Deriving DG categories.  \emph{Ann.\
    Sci.\ \'{E}cole Norm.\ Sup.\ (4)} \textbf{27} (1994), 63--102. 

\bibitem{KellerICM}  B.\ Keller, On Differential Graded categories.
  In \emph{Proceedings of the International Congress of Mathematicians
    2006, Vol.\ II} (edited by Marta Sanz-Sol\'{e} et.al.).  European
Mathematical Society Publishing House, Z\"{u}rich, 2007, 151--190.

\bibitem{KellerReiten} B.\ Keller and I.\ Reiten, Acyclic Calabi-Yau
  categories, with an appendix by Michel Van den Bergh.  Preprint
  (2006).  {\tt math.RT/0610594}.

\bibitem{KrauseARZ}  H.\ Krause, Auslander-Reiten triangles
and a theorem of Zimmermann.  \emph{Bull.\ London Math.\ Soc.}
\textbf{37} (2005), 361--372.

\bibitem{KrauseAR}  H.\ Krause, Auslander-Reiten theory via Brown
representability.  \emph{$K$-Theory} \textbf{20} (2000), 331--344.

\bibitem{ReitenVandenBergh}  I.\ Reiten and M.\ Van den Bergh, 
Noetherian hereditary abelian categories satisfying Serre duality. 
\emph{J.\ Amer.\ Math.\ Soc.} \textbf{15} (2002), 295--366.

\bibitem{Ringel}  C.\ M.\ Ringel, \emph{Tame algebras and quadratic
    forms}.  Lecture Notes in Math.\ 1099, Springer, Berlin, 1984. 

\bibitem{Schmidt}  K.\ Schmidt, Auslander-Reiten theory for simply
  connected differential graded algebras.  Thesis, University of
  Paderborn, Paderborn, 2007.  {\tt math.RT/0801.0651}.

\bibitem{Spaltenstein}  N.\ Spaltenstein, Resolutions of unbounded
  complexes.  \emph{Compositio Math.} \textbf{65} (1988), 121--154.

\end{thebibliography}
\end{document}